\documentclass[11pt,x11names]{article}

\usepackage{latexsym}
\usepackage[dvips]{graphicx}
\usepackage{amsfonts,amssymb,amsmath}
\usepackage{pdfsync}
\usepackage{epsfig}
\usepackage{geometry}
\usepackage{color}
\usepackage{pstricks}
\usepackage{setspace}
\usepackage{comment}
\usepackage{enumerate}
\usepackage{hyperref}

\newtheorem{theorem}{Theorem}
\newtheorem{proposition}[theorem]{Proposition}
\newtheorem{lemma}[theorem]{Lemma}
\newtheorem{corollary}[theorem]{Corollary}

\newtheorem{definition}{Definition}

\newtheorem{remark}[theorem]{Remark}
\newtheorem{ejem}{Example}

\newcommand{\N}{\mathbb{N}}
\newcommand{\NN}{\mathbb{N}}

\newcommand{\R}{\mathbb{R}}

\newcommand{\KK}{\mathcal{K}}

\newcommand{\Rcupinf}{(-\infty,+\infty]}

\newcommand{\hone}{({\bf H1})}
\newcommand{\htwo}{({\bf H2})}

\newcommand{\cali}{\mathcal}

\newcommand{\f}{\varphi}
\newcommand{\e}{\varepsilon}

\newenvironment{proof}{\paragraph{Proof.}}{\hfill$\square$}

\DeclareMathOperator{\argmin}{argmin}

\DeclareMathOperator{\dist}{dist}
\DeclareMathOperator{\length}{length}
\DeclareMathOperator{\dom}{dom}
\DeclareMathOperator{\Prox}{prox}
\DeclareMathOperator{\prox}{prox}
\DeclareMathOperator{\Proj}{P}

\newcommand{\sfrac}[2]{\hbox{$\frac{#1}{#2}$}}

\title{From error bounds to the complexity of first-order descent methods for convex functions}

\date{\today}

\author{J\'{e}r\^{o}me Bolte\footnote{Toulouse School of Economics, Universit\'e Toulouse Capitole, Manufacture des Tabacs, 21 all\'{e}e de Brienne, Toulouse, France. E-mail: jerome.bolte@tse-fr.eu, Effort partially sponsored by the Air Force Office of Scientific Research, Air Force Material Command, USAF, under grant number FA9550-14-1-0056, the  FMJH Program Gaspard Monge in optimization and operations research and ANR GAGA
} \and Trong Phong Nguyen\footnote{Toulouse School of Economics, Universit\'e Toulouse Capitole, 
and  Department of mathematics, National University of Civil Engineering, 55 Giai Phong, Ha noi, Viet Nam. Email: phongnt@nuce.edu.vn} \and Juan Peypouquet\footnote{Departamento de Matem\'atica \& AM2V, Universidad T\'ecnica Federico Santa Mar\'\i a, Avenida  Espa\~na 1680, Valpara\'\i so, Chile. E-mail:juan.peypouquet@usm.cl. Work supported by FONDECYT grant 1140829; Basal Project CMM Universidad de Chile; Millenium Nucleus ICM/FIC RC130003; Anillo Project ACT-1106; ECOS-Conicyt Project C13E03; Conicyt Redes 140183; and MathAmsud Project 15MAT-02.} \and Bruce W. Suter\footnote{Air Force. Research Laboratory / RITB, Rome, NY, United States of America. E-mail: bruce.suter@us.af.mil}}

\begin{document}
\maketitle

\begin{center}
{\em
Dedicated to Jean-Pierre Dedieu who was of great inspiration to us.
}
\end{center}

$\:$

\begin{abstract}
This paper shows that error bounds can be used as effective tools for deriving complexity results for first-order descent methods in convex minimization. In a first stage, this objective led us to revisit the interplay between error bounds and the Kurdyka-\L ojasiewicz (KL) inequality. One can show the equivalence between the two concepts for convex functions having a moderately flat profile near the set of minimizers (as those of functions with H\"olderian growth). A counterexample shows that the equivalence is no longer true for extremely flat functions. This fact reveals the relevance of an approach based on KL inequality. In a second stage, we show how KL inequalities can in turn be  employed to compute new complexity bounds for a wealth of descent methods for convex problems. Our approach is completely original and makes use of a one-dimensional worst-case proximal sequence in the spirit of the famous majorant method of Kantorovich. Our result applies to a very simple abstract scheme that covers a wide  class of descent methods. As a byproduct of our study, we also provide new results for the globalization of KL inequalities in the convex framework.
  
Our main results inaugurate a simple methodology: derive an error bound, compute the desingularizing function whenever possible, identify essential constants in the descent method and finally compute the complexity using the one-dimensional worst case proximal sequence. Our method is illustrated through projection methods for feasibility problems, and through the famous iterative shrinkage thresholding algorithm (ISTA), for which we show that the complexity bound is of the form  $O(q^{k})$ where the constituents of the bound only depend on error bound constants obtained for an arbitrary least squares objective with $\ell^1$ regularization.

\end{abstract}


\maketitle

\section{Overview and main results}

\paragraph{A brief insight into the theory of error bounds.} Since Hoffman's celebrated result on error bounds for systems of linear inequalities \cite{Hof}, the study of error bounds has been successfully applied to problems in sensitivity, convergence rate estimation, and feasibility issues. In the optimization world, the first natural extensions were made to convex functions by Robinson \cite{Rob}, Mangasarian \cite{Man}, and Auslender-Crouzeix \cite{Aus}. However, the most striking discovery came  years before in the pioneering works of \L ojasiewicz \cite{Loja58,Loja59} at the end of the fifties: under a mere compactness assumption, the existence of error bounds for arbitrary continuous semi-algebraic  functions was provided. Despite their remarkable depth, these works 
remained unnoticed by the optimization community during a long period (see \cite{LuoPang}).  At the 
beginning of the nineties, motivated by numerous applications, many researchers started working 
along these lines, in quest for quantitative results that could produce more effective tools. The survey 
of Pang \cite{Pang} provides a comprehensive panorama of results obtained around this time. The 
works of Luo \cite{LuoLuo,LuoPang,LuoSturm} and Dedieu \cite{dedieu} are also important 
milestones in the theory. The recent works \cite{Li,LiMorPham,Vui,LMNP,amirbeck} provide even 
stronger quantitative results by using the powerful machinery of algebraic geometry or advanced 
techniques of convex optimization.

\paragraph{A methodology for complexity of first-order descent methods.} Let us introduce  the concepts used in this work and show how they can be arranged to devise a new and systematic approach to complexity. Let $H$ be a real Hilbert space, and let $f:H\to\Rcupinf$ be a proper lower-semicontinuous convex function achieving its minimum $\min f$ so that $\argmin f\neq \emptyset$. In its most simple version, an {\em error bound} is an inequality of the form
\begin{equation}\label{omeg}\omega\big(f(x)-\min f\big)\geq \dist (x,\argmin f),
\end{equation}
where $\omega$ is an increasing function vanishing at $0$ --called here the {\em residual function}--, and where $x$ may evolve either in the whole space or in a bounded set. {\em H\"{o}lderian} error bounds, which are very common in practice, have a simple power form
$$f(x)-\min f\geq \gamma \dist ^p(x,\argmin f),$$
with $\gamma >0$, $p\geq1$ and thus $\omega(s)=(\frac{1}{\gamma}s)^\frac{1}{p}$. When functions are semi-algebraic on $H=\R^n$ and ``regular" (for instance, continuous), the above inequality is known to hold on any compact set \cite{Loja58,Loja59}, a modern reference being \cite{coste}. This property is known in real algebraic geometry under the name of {\em \L ojasiewicz inequality}. However, since we work here mainly in the sphere of optimization and  follow complexity  purposes, we shall refer to this inequality as to the {\em \L ojasiewicz error bound inequality}.

Once the question of computing constants and exponents (here $\gamma$ and $p$) for a given minimization problem is settled (see the fundamental works \cite{LuoSturm,Li,amirbeck,Vui}), it is  natural to wonder whether these concepts are connected to the complexity  properties of first-order methods for minimizing $f$.  Despite the important success of the error bound theory in several branches of optimization, we are not aware of a solid theory connecting the error bounds we consider (as defined in \eqref{omeg}), with the study of the complexity of general descent methods. There are, however, several works connecting error bounds with the convergence rates results of first-order methods (see e.g.,  \cite{Roc_mono,LuoTseng,Ferris,Nedic,BeckTeboulle,abd,Pey}). See also the new and interesting work \cite{LMNP} that provides a wealth of error bounds and some applications to convergence rate analysis. An important fraction of these works involves ``first-order error bounds"\footnote{That is, involving inequalities of the type 
$\|\nabla f(x)\|\geq \omega(\dist(x,\argmin f))$}  (see \cite{LuoPang,LuoTseng}) that are different from those we consider here.

Our answer to the connection between complexity and ``zero-order error bounds" will partially come from a related notion, also discovered by \L ojasiewicz and further developed by Kurdyka in the semi-algebraic world: the {\em \L ojasiewicz gradient inequality}. This inequality, also called    Kurdyka-\L ojasiewicz (KL) inequality (see  \cite{BolDanLeyMaz}), asserts that for any smooth semi-algebraic function $f$ there is a smooth concave function $\f$ such that 
$$\|\nabla \left(\varphi\circ (f-\min f)\right)(x)\| \geq 1$$ 
for all $x$ in some neighborhood of the set $\argmin f$. Its generalization to the nonsmooth case \cite{BolDanLew1,BolDanLewShi07} has opened very surprising roads in the nonconvex world and  it has allowed  to perform convergence rate analyses for many important algorithms in optimization \cite{AttBolSva,BST,PierreGuiJuan}. In a first stage of the present paper we show, when $f$ is convex, that error bounds are equivalent to nonsmooth KL~inequalities provided the residual function has a {\em moderate behavior} close to 0 (meaning that its derivative blows up at reasonable rate). Our result includes, in particular, all power-type examples like the ones that are often met in practice\footnote{An absolutely crucial asset of error bounds and KL inequalities in the convex world is their global nature under a mere coercivity assumption -- see Section~\ref{s:theory}.}.

Once we know that error bounds provide a KL inequality, one still needs to make the connection with the actual complexity of first-order methods. 
This is  probably the main contribution in this paper: to any given convex objective $f:H\to\Rcupinf$ and descent sequence of the form
\begin{itemize}
	\item[(i)]  $f(x_k)+a\|x_k-x_{k-1}\|^2\le f(x_{k-1}),$
	\item[(ii)] $\|\omega_{k}\|\le b\|x_k-x_{k-1}\|$ where $\omega_{k}\in\partial f(x_k)$, $k\geq 1,$
		\end{itemize}
we associate {\em a worst case one dimensional proximal method}
$$\alpha_k=\argmin\left\{\varphi^{-1}(s)+\frac{1}{2\zeta}(s-\alpha_k)^2:s\geq 0\right\},\: \alpha_0=\varphi^{-1}(f(x_0)),$$ where $\zeta$ is a constant depending explicitly on the triplet of positive real numbers $(a,b,\ell)$ where $\ell>0$ is a Lipschitz constant of $\Big(\varphi^{-1}\Big)'$.
Our complexity result asserts, under weak assumptions that the ``1-D prox" governs the complexity of the original method through the elementary and natural inequality
$$f(x_k)-\min f\leq \varphi^{-1}(\alpha_k), \, k\geq0.$$ Similar results for the sequence are provided.  These ideas are already  present in \cite{HDR} and \cite[Section 3.2]{BBJ}. The function $\varphi^{-1}$ above --the inverse of a desingularizing function for $f$ on a convenient domain--  contains almost all the information our approach provides on the  complexity of descent methods. As explained previously, it depends on the precise knowledge of a KL inequality and thus, in this convex setting, of an error bound. The reader familiar with second-order methods might have recognized the spirit of the majorant method of Kantorovich \cite{Kan64}, where a reduction to dimension one is used to study Newton's method. 

\paragraph{Deriving complexity bounds in practice: applications.} Our theoretical results inaugurate a simple methodology: derive an error bound, compute the desingularizing function whenever possible, identify essential constants in the descent method and finally compute the complexity using the one-dimensional worst case proximal sequence.  We consider first some classic well-posed problems: finding a point in an intersection of closed convex sets with regular intersection or uniformly convex problems, and we show how complexity of some classical methods can be obtained or recovered. We revisit the {\em iterative shrinkage thresholding algorithm} (ISTA) applied to a least squares objective with $\ell^1$ regularization  \cite{daub} and we prove that its complexity is of the form $O(q^k)$ with $q\in(0,1)$ (see \cite{Nedic} for a pioneering work in this direction and also  \cite{jalal} for further geometrical insights). This result contrasts with what was known on the subject \cite{BT08,drori} and suggests that many questions on the complexity of first-order methods remain open.

\paragraph{Theoretical aspects and complementary results.} As explained before, our paper led us to establish several theoretical results and to clarify some questions appearing in a somehow disparate manner in the literature. We first explain how to pass from error bounds to KL inequality in the general setting of Hilbert spaces and vice versa, similar questions appear in \cite{BolDanLew1,LiMorPham,LMNP}. This result is proved by considering the interplay between the contraction semigroup generated by the subdifferential function and the $L^1$ contraction property of this flow. These results are connected to the geometry of the residual functions $\omega$ and break down when error bounds are too flat. This is shown in Section~\ref{s:theory} by a dimension 2 counterexample presented in \cite{BolDanLeyMaz} for another purpose.

Our investigations also led us to consider the problem of KL inequalities for convex functions, a problem partly tackled in \cite{BolDanLeyMaz}. We show how to extend convex KL inequalities from a level set to the whole space. We also show that   compactness and semi-algebraicity ensure that real semi-algebraic or definable coercive convex functions are automatically KL {\em on the whole space}. This result has an interesting theoretical consequence in terms of complexity: {\em abstract descent  methods for coercive semi-algebraic convex problems are systematically amenable to a full complexity analysis provided that a desingularizing function --known to exist-- is explicitly computable}.

\paragraph{Organization of the paper.} Section 2 presents the basic notation and concepts used in this paper, especially concerning elementary convex analysis, error bounds and KL inequalities. Readers familiar with the field can directly skip to Section 3, devoted to the equivalence between KL  inequalities and error bounds. We also give some examples where this equivalence is explicitly exploited. Section~\ref{s4} establishes complexity results using KL inequalities, while  Section~\ref{s5} provides illustrations of our general methodology for the $\ell^1$ regularized least squares method and feasibility problems.  Finally, Section~\ref{s:theory} contains further theoretical aspects related to our main results, namely: some counterexamples to the equivalence between error bounds and KL inequalities, more insight into the relationship between KL inequalities and the length of subgradient curves,  globalization of KL inequalities and related questions.









\if{
{\color{blue}
In 1963, \L ojasiewicz \cite{Loja63} proved that the variation of the values of a real analytic function in a neighborhood of a critical point can be bounded by a power of the norm of its gradient. The resulting estimation has proved to be a powerful tool to analyze several issues concerning, for instance, the long-term behavior of evolution equations \cite{Sim83,HuaTak,ChiJen,HarJen,Har}, and the convergence of several keynote algorithms in optimization \cite{AbsMahAnd,attbol,AttBolRedSou,AttBolSva}.

Besides real analytic ones, there are other classes of functions that satisfy either \L ojasiewicz's original inequality or a generalization stated by Kurdyka \cite{Kur98}. Semi- and sub-analytic functions have this property \cite{Loj93,KurPar,BolDanLew1}. See \cite{BolDanLewShi07} for functions defined on o-minimal structures. The convex case is studied in \cite{BolDanLeyMaz}.

The purpose of this work is to investigate  further relationships between \L ojasiewicz's property for convex functions and:
\begin{itemize}
\item Boundedness of the length of subgradient trajectories;
\item Boundedness of the length of piecewise linear interpolation of sequences generated by abstract descent methods;
\item Error bounds; and
\item Properties of the Fenchel conjugate of the function.
\end{itemize}

The paper is organized as follows: 

}

}\fi


\section{Preliminaries}

In this section, we recall the basic concepts, notation and some well-known results to be used throughout the paper. In what follows, $H$ is a real Hilbert space and $f:H\to\Rcupinf$ is proper, lower-semicontinuous and convex. We are interested in some properties of the function $f$ around the set of its minimizers, which we suppose to be nonempty and denote by $\argmin f$ or $S$. We assume, without loss of generality, that $\min f=0$.

\subsection{Some convex analysis}\label{SS:convex_analysis}

We use the standard notation from \cite{Rockafellar} (see also \cite{BauCom,Pey_book} and \cite{morduk}). The {\em subdifferential} of $f$ at $x$ is defined as
$$\partial f(x)=\{u\in H:f(y)\ge f(x)+\langle u,y-x\rangle\hbox{ for all }y\in H\}.$$
Clearly, $\hat x$ minimizes $f$ on $H$ if, and only if, $0\in\partial f(\hat x)$. The {\em domain} of the point-to-set operator $\partial f:H\rightrightarrows H$ is $\dom \partial f:=\{x\in H:\partial f(x)\neq\emptyset\}.$ For $x\in\dom \partial f$, we denote by $\partial^0f(x)$ the least-norm element of $\partial f(x)$. The vector  $\partial^0f(x)$ exists and is unique as it is the projection of $0\in H$ onto the nonempty closed convex set $\partial f(x)$. We have $\|\partial^0f(x)\|=\dist(0,\partial f(x))$ (when $x$ is not in $\dom \partial f$ we set $\|\partial^0f(x)\|=+\infty$). We adopt the convention $s\times (+\infty)=+\infty$ for all $s>0$.
 
\noindent
Given $x\in H$, the function $f_x$, defined by
$$f_x(y)=f(y)+\frac{1}{2}\|y-x\|^2$$
for $y\in H$, has a unique minimizer, which we denote by $\prox_f(x)$. Using Fermat's Rule and the Moreau-Rockafellar Theorem, $\prox_f(x)$ is characterized as the unique solution of the inclusion
$x-\prox_f(x)\in \partial f\left(\prox_f(x)\right).$ 
In particular, $\prox_f(x)\in \dom\partial f\subset\dom f\subset H$. The mapping $\prox_f:H\to H$ is the {\em proximity operator} associated to $f$. It is easy to prove that $\prox_f$ is  Lipschitz  continuous with constant 1.

\begin{ejem}{\rm
If $C\subset H$ is nonempty, closed and convex, the {\em indicator function} of $C$ is the function $i_C:H\to(-\infty,\infty]$, defined by
$$i_C(x)=\begin{cases}
	0& \text{ if } x\in C\\
	+\infty&\text{ otherwise.}
	\end{cases}$$
It is proper, lower-semicontinuous and convex. Moreover, for each $x\in H$, $\partial i_C(x)=N_C(x)$, the {\em normal cone} to $C$ at $x$. In turn, $\prox_{i_C}$ is the {\em projection operator} onto $C$, which we denote by $P_C$.}
\end{ejem}

\if{


Let $g,h:H\to\Rcupinf$. The {\em infimal convolution} of $g$ and $h$ is the function $g\,\square\,h:H\to\Rcupinf$ defined by
$$ g\,\square\,h\,(x)=\inf_{y\in H} \{g(y)+h(x-y)\}.$$
If $g$ and $h$ are proper, lower semicontinuos and convex, then so is $g\,\square\,h$.

\begin{ejem}
Let $C\subset H$ be nonempty, closed and convex. For each $p>0$, 
$$\left(i_C\,\square\, \|\cdot \|^p\right)(x)=\inf_{y\in \R^n}\{i_C(y)+\|x-y\|^p\}=\dist(x,C)^p.$$
\end{ejem}

Let $f:H\to\Rcupinf$ be proper, lower-semicontinuous and convex. For $\lambda>0$, the function $f_\lambda:H\to\R$ defined by 
$$f_\lambda(x)=\left(f\,\square\,\displaystyle{\frac{1}{2\lambda}}\|\cdot\|^2\right)(x)=\inf_{y\in H}\left\{f(y)+\frac{1}{2\lambda}\|x-y\|^2\right\}$$ is the {\em Moreau envelope} of $f$ with parameter $\lambda$. This function is a convex function finite everywhere and which satisfies $$\text{$f_\lambda\le f$, $\inf_H(f_\lambda)=\inf_H(f)$, and $\argmin(f_\lambda)=\argmin f$,}$$ for all $\lambda>0$ (see e.g. \cite{BauCom}). Moreover, by strong convexity, the infimum in the definition of $f_\lambda$ is attained at a unique point denoted by $\Prox_{\lambda f}(x)$, and characterized by the inclusion
$$x-\Prox_{\lambda f}(x)\in\lambda\,\partial f(\Prox_{\lambda f}(x)).$$
The mapping $\Prox_{\lambda f}:H\to H$, known as the (Moreau) {\em proximity operator} for $f$ with parameter $\lambda$, is nonexpansive. Finally, the function $f_\lambda$ is differentiable in the sense of Fr\'echet, and its gradient, given by 
$$\nabla f_\lambda(x)=\frac{x-\Prox_{\lambda f}(x)}{\lambda}, \:\forall x\in H,$$ is  Lipschitz  continuous with constant $\frac{1}{\lambda}$.

\begin{ejem}
Let $C\subset H$ be nonempty, closed and convex, and let $f=i_C$ be the indicator function of $C$. Then $f_\lambda(x)=\frac{1}{2\lambda}\dist(x,C)^2$ and $\Prox_{\lambda f}(x)=\Proj_C(x)$ for all $\lambda>0$.
\end{ejem}

The {\em Fenchel conjugate} of $f$ is the function $f^*:H\to\Rcupinf$ defined by
$$f^*(u)=\sup_{x\in H}\{\langle x,u\rangle-f(x)\}.$$
It is lower-semicontinuous and convex. Moreover, $f:H\to\Rcupinf$ is proper, lower-semicontinuous and convex if, and only if, $f^{**}=f$. The Fenchel conjugation is order-reversing. More precisely, if $g,h:H\to\Rcupinf$ are proper, lower-semicontinuous and convex, then $g\le h$ if, and only if, $h^*\le g^*$.

{\color{red} \begin{ejem}
If $f=i_C$, then $f^*=\sigma_C$, the {\em support function} of $C$.
\end{ejem}
}

\begin{ejem}
Let $f=r\|\cdot \|^p$ with $r>0$ and $p\ge1$. If $p=1$, then $f^*=\delta_{B(0;r)}$. For $p>1$, we have $f^*=\tilde r\|\cdot \|^q$, where $q=\displaystyle\frac{p}{p-1}$ and $\tilde r=r(p-1)(rp)^{-q}$.
\end{ejem}

There is an important relationship between the Fenchel conjugate and the infimal convolution, namely
$$\left(g\,\square\, h\right)^*=g^*+h^*.$$

\begin{ejem}\label{EJ:distance} Let $C$ be a nonempty, closed and convex subset of $H$, $r>0$ and $p\geq 1$. Then
$$
\left(\sfrac{}{}\!r\dist(\cdot,C)^p\right)^* =(i_C\,\square\, r\|\cdot \|^p)^*=i_C^*+(r\|\cdot\|^p)^*
	=\begin{cases}
	\sigma_C+\delta_{B(0;r)}& \text{ if }p=1\\
	\sigma_C+ \tilde K\|\cdot\|^q& \text{ if } p>1,	
	\end{cases}
$$
where $q=\displaystyle\frac{p}{p-1}$ and $\tilde r=r(p-1)(rp)^{-q}$.
\end{ejem}


}\fi

\subsection{Subgradient curves}

Consider the differential inclusion
$$\begin{cases}
\dot y(t)\in -\partial f(y(t)), \quad\text{for almost all $t$ in }(0,+\infty)\\
y(0)=x,
\end{cases}$$
where $x\in\overline{\dom f}$ and $y(\cdot)$ is an absolutely continuous curve in $H$. The main properties of this system $-$ for the purpose of this research $-$ are summarized in the following:

\begin{theorem}[Br\'ezis \cite{HaimBrezis}, Bruck \cite{Bruck}]\label{P:gradient_curve}
For each $x\in \overline{\dom f}$, there is a unique absolutely continuous curve $\chi_x:[0,\infty)\to H$ such that $\chi_x(0)=x$ and
$$\dot{\chi}_x(t)\in -\partial f\left(\chi_x(t)\right)$$
for almost every $t>0$. Moreover,
\begin{itemize}
\item[i)] $\displaystyle\sfrac{d}{dt}\chi_x(t^+)=-\partial^0f(\chi_x(t))$ for all $t>0$;
\item[ii)] $\displaystyle\sfrac{d}{dt}f\left(\chi_x(t^+)\right)=-\|\dot{\chi}_x(t^+)\|^2$ for all $t>0$;
\item [iii)] For each $z\in S$, the function $t\mapsto\|\chi_x(t)-z\|$ decreases;
\item [iv)] The function $t\mapsto f(\chi_x(t))$ is nonincreasing and $\lim_{t\to\infty}f(\chi_x(t))=\min f$;
\item [v)] $\chi_x(t)$ converges weakly to some $\hat x\in S$, as $t\to\infty$.
\end{itemize}
 \end{theorem}

The proof of the result above is provided in \cite{HaimBrezis}, except for part ${\rm v)}$, which was proved in \cite{Bruck}. The trajectory $t\mapsto \chi_x(t)$ is  called a {\em subgradient curve.}

\subsection{Kurdyka-\L ojasiewicz inequality}
In this subsection, we present the nonsmooth Kurdyka-\L ojasiewicz inequality introduced in \cite{BolDanLew1} (see also \cite{BolDanLewShi07,BolDanLeyMaz}, and the fundamental works \cite{Loja63,Kur98}). To simplify the notation, we write $[f<\mu]=\{x\in H:f(x)<\mu\}$ (similar notation can be guessed from the context). Let $r_0>0$ and set
$$\cali K(0,r_0)=\left\{\, \f \in C^0[0,r_0)\cap C^1(0,r_0),\ \f(0)=0,\ \f\text{ is concave and }\f'>0\, \right\}.$$

The function $f$ satisfies the {\em Kurdyka-\L ojasiewicz {\rm (KL)} inequality} (or has the KL {\em property}) locally at $\bar x\in \dom f$ if there exist $r_0>0$, $\f\in \cali K(0,r_0)$ and $\varepsilon >0$ such that
$$
\f'\left(f(x)-f(\bar x)\right)\,\dist(0,\partial f(x))\geq 1
$$
for all $x\in B(\bar x,\e)\cap [f(\bar x)<f(x)<f(\bar x)+r_0]$. We say $\f$ is a {\em desingularizing function} for $f$ at $\bar x$. This property basically expresses the fact that a function can be made sharp by a reparameterization of its values. 
 
If $\bar x$ is not a minimizer of $f$, the KL inequality is obviously satisfied at $\bar x$. Therefore, we focus on the case when $\bar x\in S$. Since $f(\bar x)=0$, the KL inequality reads
\begin{equation}\label{l1'}
\f'\left(f(x)\right)\,\|\partial^0 f(x)\|\geq 1
\end{equation}
for $x\in B(\bar x,\e)\cap [0<f<r_0$]. The function $f$ has the KL  property on $S$ if it does so at each point of $S$.
 
The {\em \L ojasiewicz gradient inequality} corresponds to the case when $\f(s)=c s^{1-\theta}$ for some $c>0$ and $\theta\in [0,1)$. 
Following \L ojasiewicz original presentation, (\ref{l1'}) can be reformulated as follows
$$\|\partial ^0 f(x)\|\geq c'\, f(x)^\theta,$$
where $c'=[(1-\theta)c]^{-1}$. The number $\theta$ is the {\em \L ojasiewicz exponent}. If $f$ has the KL property and admits the same desingularizing function $\varphi$ at {\em every point}, then we say that $\varphi$ is a {\em global} desingularizing function for $f$.

KL inequalities were developed within the fascinating world of real  semi-algebraic sets and functions. For that subject, we refer the reader to the book \cite{coste} by Bochnak-Coste-Roy.  

We recall the following theorem on the nonsmooth KL inequality (which follows the pioneering works of \L ojasiewicz \cite{Loja63} and Kurdyka \cite{Kur98}). It is one of the cornerstones of the present research:

\begin{theorem}[Bolte-Daniilidis-Lewis \cite{BolDanLew1}]{\rm(}Nonsmooth KL inequality{\rm)} If $f:\R^n\to\Rcupinf$ is proper, convex, lower-semicontinuous and semi-algebraic\footnote{If {\em semi-algebraic} is replaced by {\em subanalytic} or {\em definable}, we obtain the same results.}, then it has the KL property around each point in $\dom f$. \end{theorem}

Under an additional coercivity assumption, a global result is provided in Subsection \ref{locglob}.

\subsection{Error bounds}\label{SS:EB}

Consider a nondecreasing function $\omega:[0,+\infty[ \to [0,+\infty[$ with $\omega(0)=0$. The function $f$ satisfies a local error bound with {\em residual function} $\omega$ if there is $r_0>0$ such that
$$(\omega\circ f)(x)\geq \dist(x,S)$$
for all $x\in [0\le f\leq r_0]$  (recall that $\min f=0$). Of particular importance is the case when $\omega(s)=\gamma^{-1} s^{\frac{1}{p}}$ with $\gamma>0$ and $p\ge 1$, namely:
$$f(x)\geq \gamma\,\dist(x,S)^p$$
for all $x\in [0\le f\leq r_0]$. 

If $f$ is convex lower semicontinuous, we can extend the error bound beyond $[0\le f\leq r_0]$ by  linear extrapolation. More precisely, let $x\in \dom f$ such that $f(x)>r_0$. Then $f$ is continuous on the segment  $[x,P_S(x)]$. Therefore, there is $x_0\in [x,P_S(x)]$ such that 
$f(x_0)=r_0$.  By convexity, we have
$$\frac{f(x)-0}{\dist(x,S)}\ge\frac{f(x_0)-0}{\dist(x_0,S)}\ge r_0\left(\frac{\gamma}{r_0}\right)^{\frac{1}{p}}.$$
It follows that 
$$\begin{array}{rclcl}
f(x) & \ge & \gamma\,\dist(x,S)^p & \hbox{for} & x\in [0\le f\leq r_0],\\
f(x) & \ge & r_0^{\frac{p-1}{p}}\,\gamma^{\frac{1}{p}}\,\dist(x,S) & \hbox{for} & x\notin[0\le f\leq r_0].
\end{array}$$
This entails that
$$f(x)+f(x)^{\frac{1}{p}}\ge \gamma_0\,\dist(x,S)$$
for all $x\in H$, where $\gamma_0=\left(1+r_0^{\frac{p-1}{p}}\right)\,\gamma^{\frac{1}{p}}$. This is known in the literature as a global {\em H\"{o}lder-type} error bound (see \cite{Li}). Observe that it can be put under the form $\omega(f(x))\geq \dist(x,S)$ by simply setting 
$\omega(s)=\frac{1}{\gamma_0}(s+s^{\frac{1}{p}})$.
When combined with the \L ojasiewicz error bound inequality \cite{Loja58,Loja59}, the above remark implies immediately the following result:

\begin{theorem}[Global error bounds for semi-algebraic coercive convex functions]\label{GBS} $\:$\\
Let $f:\R^n\to\Rcupinf$ be proper, convex, lower-semicontinuous and semi-algebraic, and assume that $\argmin f$ is nonempty and compact. 
Then $f$ has a global error bound
$$f(x)+f(x)^{\frac{1}{p}}\ge  \gamma_0\,\dist(x,\argmin f), \forall x\in \R^n,$$
where $\gamma_0>0$ and $p\geq1$ is a rational number.
\end{theorem}


\section{Error bounds with moderate growth are equivalent to \L ojasiewicz inequalities}\label{s:holder}

In this section, we establish a general equivalence result between error bounds and  KL  inequalities. Our main goal is to provide a simple and natural way of explicitly computing \L ojasiewicz exponents and, more generally, desingularizing functions. To avoid perturbing the flow of our general methodology on complexity, we discuss limitations and extensions of our results later, in Section~\ref{s:theory}. 

As shown in Section~\ref{s4}, KL inequalities allow us to derive complexity bounds for first-order methods. However, KL inequalities with known constants are in general difficult to establish while error bounds are more tractable (see e.g., \cite{Li} and references therein). The fact that these two notions are equivalent opens a wide range of possibilities when it comes to analyzing algorithm complexity.

\subsection{Error bounds with moderate residual functions and \L ojasiewicz inequalities}


\noindent{\bf Moderate residual functions.}  Error bounds often have a power or H\"older-type form (see e.g. \cite{LuoPang,LuoLuo,LuoSturm,Li,NgZheng,Vui}). They can be either very simple $s\to as^p$ or exhibit two regimes, like for instance, $s\to as^p+bs^q$. In any cases, for all concrete instances we are aware of, residual functions are systematically semi-algebraic or of ``power-type". In this paper, we introduce a category of functions that allows to encompass these semi-algebraic cases and even more singular ones into a simple and appealing framework. A function $\varphi: [0,r)\to \R$ in $C^1(0,r)\cap C^0[0,r)$ and vanishing at the origin, has a {\em moderate behavior (near the origin)} if it satisfies a differential equation of the type 
$$s\varphi'(s)\geq c\varphi(s), \:\forall s\in (0,r),$$
 where $c$ is a positive constant (observe that by concavity one has necessarily $c\leq 1$).
A pretty direct use of the Puiseux Lemma (see \cite{coste}) shows:

\begin{lemma} If $\varphi: [0,r)\to \R$ in $C^1(0,r)\cap C^0[0,r)$, vanishes at the origin and is semi-algebraic or subanalytic then it has a moderate behavior.
 \end{lemma}
 \smallskip
 
The following theorem asserts that if $\varphi$ has a moderate behavior, $f$ has the global KL  property if, and only if, $f$ has a global error bound. Besides, the desingularizing function in the KL  inequality and the residual function in the error bound are essentially the same, up to a multiplicative constant. As explained through a counterexample in subsection~\ref{locglob}, the equivalence breaks down if one argues in a setting where the derivative $\varphi$ can blow up faster. This result is related to results obtained in  \cite{BolDanLeyMaz,BolDanLew1,corv,LiMorPham,LMNP} and also shares some common techniques.

\begin{theorem}[Characterization of \L ojasiewicz inequalities for convex functions]\label{Theolocal}$\:$\\
Let $f:H\to\Rcupinf$ be a proper, convex and lower-semicontinuous, with $\min f=0$. Let $r_0>0$, $\varphi \in \KK (0,r_0)$, $c>0$, $\rho>0$ and $\bar x\in\argmin f$.
\begin{enumerate}[(i)]	
	\item {\rm [KL inequality implies error bounds]}  If $\varphi'\left(f(x)\right)\|\partial^0 f(x)\|\geq 1$ for all $x\in [0<f<r_0]\cap B(\bar x,\rho)$, then $\dist(x,S)\leq \varphi\left(f(x)\right)$ for all $x\in [0<f<r_0]\cap B(\bar x,\rho)$.
	\item {\rm [Error bounds implies KL inequality]}  Conversely, if $s\varphi'(s)\geq c\varphi(s)$ for all $s\in (0,r_0)$ ($\varphi$ has a moderate behavior), and $\varphi(f(x))\geq \dist(x,S)$ for all $x\in [0<f<r_0]\cap B(\bar x,\rho)$, then $\varphi'\left(f(x)\right)\|\partial^0 f(x)\|\geq c$ for all $x\in [0<f<r_0]\cap B(\bar x,\rho)$.
\end{enumerate}
\end{theorem}

\begin{proof} (i) Recall that the mapping $[0,+\infty)\times \overline{\dom f}\ni (t,x)\to \chi_x(t)$ denotes the semiflow associated to $-\partial f$ (see previous section).   
Since $f$ satisfies Kurdyka-\L ojasiewicz inequality, we can apply Theorem~\ref{P:1} of Section~\ref{s:theory}, to obtain $$\|\chi_x(t)-\chi_x(s)\|\le \varphi(f(\chi_x(t)))-\varphi(f(\chi_x(s))),$$
for each $x\in B(\bar x,\rho)\cap[0<f\le r_0]$ and $0\le t<s$. 
As  established in Theorem~\ref{P:1}, $\chi_x(s)$ must converge strongly to some $\tilde x\in S$ as $s\to\infty$. Take $t=0$ and let $s\to\infty$ to deduce that $\|x-\tilde x\|\le \varphi(f(x))$. Thus $\varphi( f(x))\ge \dist(x,S)$.\\
(ii) Take $x\in[0<f< r_0]\cap B(x,\rho)$ and write $y=P_S(x)$. By convexity, we have
$$0=f(y)\ge f(x)+\langle \partial^0 f(x),y-x\rangle.$$
This implies
$$f(x)\le \|\partial^0 f(x)\|\,\|y-x\|=\dist(x,S)\|\partial^0 f(x)\|\le\varphi (f(x))\|\partial^0 f(x)\|.$$
Since $f(x)>0$, we deduce that
$$1\leq \|\partial^0 f(x)\|\frac{\varphi (f(x))}{f(x)}\leq \frac{1}{c}\|\partial^0 f(x)\|\, \varphi '(f(x)),$$
and the conclusion follows immediately.
\end{proof}
\medskip

In a similar fashion, we can characterize the global existence of a \L ojasiewicz gradient inequality.

\begin{corollary}{\bf(Characterization of \L ojasiewicz inequalities for con\-vex func\-tions: glo\-bal ca\-se)}\label{Theoglobal} 
Let $f:H\to\Rcupinf$ be a proper, convex and lower-semicontinuous, with $\min f=0$. Let $\varphi \in \KK (0,+\infty)$ and $c>0$.
\begin{enumerate}[(i)]	
\item If $\varphi'\left(f(x)\right)\|\partial^0 f(x)\|\geq 1$ for all $x\in [0<f]$, then $\dist(x,S)\leq \varphi\left(f(x)\right)$ for all $x\in [0<f]$.
\item  Conversely, if $s\varphi'(s)\geq c\varphi(s)$ for all $s\in (0,r_0)$ ($\varphi$ has moderate behavior), and $\varphi(f(x))\geq \dist(x,S)$ for all $x\in [0<f]$, then $\varphi'\left(f(x)\right)\|\partial^0 f(x)\|\geq c$ for all $x\in [0<f]$.
\end{enumerate}
\end{corollary}

\begin{remark}{\rm (a) Observe the slight dissymmetry between the conclusions of (i) and (ii) in Theorem \ref{Theolocal} and Corollary \ref{Theoglobal}: while a desingularizing function provides directly an error bound in (i), an error bound (with moderate growth) becomes desingularizing after a rescaling, namely $c^{-1}\varphi$. \\
(b) (H\"olderian case) When in $(ii)$ one has $\varphi(s)=\gamma s^{\frac{1}{p}} $ with $p\geq 1$, $\gamma>0$, then the constant $c$ is given by 
\begin{equation}\label{c}
c=\frac{1}{p}.
\end{equation}}
\end{remark}

Analytical aspects linked with the above results, such as connections with  subgradient curves and nonlinear bounds, are discussed in a section devoted to further theoretical aspects of the interplay between KL  inequality and error bounds. We focus here on the essential consequences we expect in terms of algorithms and complexity. With this objective in mind, we first provide some concrete examples in which 	a  KL  inequality with known powers and/or constants can be provided.

\subsection{Examples: computing \L ojasiewicz exponent through error bounds}\label{EBKL}

The method we use for computing \L ojasiewicz exponents is quite simple: we derive an error bound for $f$ with as much information as possible on the constants, and then we use the convexity along with either Theo\-rem \ref{Theolocal} or Co\-rol\-la\-ry \ref{Theoglobal} to compute a desingularizing function together with a domain of desingularization; this technique appears also in \cite{LiMorPham} a paper which only came to our knowledge during the finalization of our article.

\subsubsection{KL inequality for piecewise polynomial convex functions and least squares objective with $\ell^1$ re\-gu\-larization} \label{s:pol} Here, a conti\-nuous func\-tion $f:\R^n\to\R$ is {\em piecewise  polynomial} if there is a partition of $\R^n$ into finitely many polyhedra\footnote{Usual definitions allow the subdomains to be more complex} $P_1,\dots,P_k$, such that $f_i=f\vert_{P_i}$ is a   polynomial for each $i=1,\dots ,k$. The degree of $f$ is defined as $\deg(f)=\max\{\deg(f_i):i=1,\dots,k\}$.  We have the following interesting result from Li \cite[Corollary 3.6]{Li}:

\begin{proposition}[Li \cite{Li}]\label{EB_Poly}
Let $f:\R^n\to \R$ be a piecewise polynomial convex function with $\argmin f\neq\emptyset$. Then, for each $r\geq \min f$, there exists $\gamma_r>0$ such that\
\begin{equation}\label{EBba}f(x)-\min f\geq \gamma_r\,\dist\left(\sfrac{}{}\!x,\argmin f\right)^{(\deg(f)-1)^n+1}\end{equation}
for all $x\in[f\leq r]$.
\end{proposition}

Combining Proposition \ref{EB_Poly} and Corollary \ref{Theoglobal}, the above implies:

\begin{corollary}\label{C:KL_sparse}
Let $f:\R^n\to\R$ be a piecewise polynomial convex function with $\argmin f\neq\emptyset$. Then $f$ has the \L ojasiewicz property on $[f\leq r]$, with exponent $\displaystyle\theta=1-\frac{1}{(\deg(f)-1)^n+1}$.
\end{corollary}

\medskip

\noindent
{\bf Sparse solutions of inverse problems.} Let $f:\R^n\to\R$ be given by 
$$f(x)=\frac{1}{2}\|Ax-b\|_2^2+\mu\|x\|_1,$$
where $\mu>0$, $b\in\R^m$ and $A$ is a matrix of size $m\times n$. Then $f$ is obviously a piecewise polynomial convex function of degree $2$. Since $f$ is also coercive, we have $S=\argmin f\neq\emptyset$. A direct application of Proposition \ref{EB_Poly} and Corollary \ref{C:KL_sparse} gives that $f-\min f$ admits $\theta=\frac{1}{2}$ as a \L ojasiewicz exponent.

Yet, in order to derive proper complexity bounds for ISTA we need to identify a computable constant~$\gamma_r$ in  \eqref{EBba}. For this we shall apply a recent result from Beck-Shtern \cite{amirbeck}.

\medskip

First let us recall some basic results on error bounds (see e.g., \cite{Hof,zal}).  In what follows, $\|M\|$ denotes the {\em spectral} or {\em operator} norm of a real matrix $M$.\footnote{It is the largest singular value of $M$, which is the square root of the largest eigenvalue of the positive-semidefinite square matrix $M^TM$, where $M^T$ is the transpose matrix of $M$.}

\begin{definition}[Hoffman's error bound]\label{d:hof} Given positive integers $m,n,r$, let $A\in \R^{m\times n}$, $a\in \R^m$, $E=\R^{r\times n}$, $e\in \R^r$. We consider the two polyhedra
$$X=\left\{x\in \R^n: Ax\leq a\right\}, \: Y=\left\{x\in \R^n: Ex=e\right\},$$
and we assume that $X\cap Y\neq \emptyset$. There exists a constant $\nu=\nu(A,E)\geq 0$, that only depends on the pair $(A,E)$ and is known as {\em Hoffman's constant} for the pair $(A,E)$, such that 
\begin{equation}\label{hof}
\dist(x,X\cap Y)\leq \nu \|Ex-e\|, \:\forall x\in X.
\end{equation}
\end{definition}
A crucial aspect of Hoffman's  error bound is the possibility of estimating the constant $\nu$ from the data $A,E$. We will not enter into these details here, we simply refer the reader  to the work of Z\v{a}linescu \cite{zal} and the references therein.  

As suggested by Beck, we shall now apply a very useful result from \cite{amirbeck} to derive an error bound for~$f$. 
Recall that $S=\argmin_{\R^n} f$ is convex, compact and nonempty. For any $x^* \in S$,  $f(x^*)\leq f(0)=\frac{1}{2}\|b\|^2$ which implies $\|x^*\|_1\leq \frac{\|b\|^2}{2\mu }$. Hence $S\subset \left\lbrace x\in \R^n: \|x\|_1\leq R\right\rbrace$ for any fixed $R> \frac{\|b\|^2}{2\mu}$. For such a bound $R$, one has 
\begin{align}
\min_{\R^n} f =&\min \left\{ \frac{1}{2}\|Ax-b\|^2_2+\mu \|x\|_1\,:x\in\R^n\right\}\notag\\
=&\min \left\{\frac{1}{2}\|Ax-b\|^2_2+\mu y \: : (x,y)\in \R^n\times\R, \,\|x\|_1\leq R, \, y=\|x\|_1 \right\}\notag\\
=&\min \left\{ \frac{1}{2}\|Ax-b\|^2_2+\mu y  \: : (x,y)\in \R^n\times\R, \,\|x\|_1-y\leq 0,\, y\leq R  \right\}\notag\\
=&\min \left\{ \frac{1}{2}\|\tilde{A}\tilde{x}-\tilde{b}\|^2_2+\langle \tilde{\mu},\tilde{x}\rangle \: : ~\tilde{x}=(x,y)\in \R^{n}\times\R,\: M\tilde{x}\leq \tilde{R} \right\} \label{form}
\end{align}
 where $$\left\{\begin{aligned}
  & \bullet \tilde{A}=[A,0_{\R^{m\times1}}]\in \R^{m\times (n+1)},~ \tilde{b}=(b_1,\ldots,b_m,0)\in \R^{m+1},\\ 
  &\bullet \tilde{\mu}=(0,\ldots,0,\mu)\in \R^{n+1}, ~\tilde{R}=(0,\ldots, 0,R)\in \R^{n+1}\\
 & \bullet M= \begin{bmatrix}
 E & \:\quad-1_{\R^{2^n\times 1}} \\ 
 \quad \:\: 0_{\R^{1\times n}} & 1
 \end{bmatrix} \text{ is a matrix of size }(2^n+1)\times (n+1),\\
 & \text{where $E$ is a matrix of size $2^n\times n$ whose rows are all possible distinct vectors of size $n$}\\
 & \text{  of the form $e_i=(\pm 1,\ldots,\pm 1)$ for all $i=1,\ldots,2^n$. The order of the $e_i$ being arbitrary.}
 \end{aligned}\right.$$
 Set $\tilde{X}:=\left\lbrace \tilde{x}\in \R^{n+1}:M\tilde{x}\leq \tilde{R}\right\}$ and observe that the ``geometrical complexity" of the problem is now embodied in the matrix $M$.\\
  It is clear that  $$(x^*,y^*)\in \tilde{S}:=\mathop{\argmin}\limits_{\tilde x\in \tilde{X}}\left( \tilde{f}(\tilde{x}):= \frac{1}{2}\|\tilde{A}\tilde{x}-\tilde{b}\|^2_2+\langle \tilde{\mu},\tilde{x}\rangle\right) \text{ if and only if }\left(x^*\in S\mbox{ and }y^*=\|x^*\|_1\right).$$
 Using  \cite[Lemma 2.5]{amirbeck}, we obtain:   \begin{align}
\dist^2 (\tilde{x},\tilde{S})\leq &\:\,\nu^2 \left( \|\tilde{\mu}\|D+3GD_A+2G^2+2\right) \left(\tilde{f}(\tilde{x})-\tilde{f}(\tilde{x}^*)\right), \: \forall \tilde{x}\in \tilde{X}\notag
 \end{align}
where 
\begin{itemize}
\item $\tilde{x}^*=(x^*,y^*)$ is any optimal point in $\tilde{S}$, 
\item  $\nu$ is the Hoffman constant associated with the couple $(M,[\tilde{A}^T,\tilde\mu^T]^T)$ as in Definition~\ref{d:hof} above.
\item $D$ is the Euclidean diameter of the polyhedron $\tilde{X}=\{(x,y)\in\R^{n+1}: \|x\|_1\leq y\leq R\}$ and is thus the maximal distance between two vertices. Hence $D= 2R$.
\item $G$ is the maximal Euclidean norm of the gradient of  $\frac{1}{2}\|.-\tilde{b}\|^2$ over $\tilde{A}(\tilde{X})$, hence, $G\leq R\|A\|+\|b\|$.
\item $D_A$ is the Euclidean diameter of the set $\tilde{A}(\tilde{X})$, thus $ D_A=\max_{x_i\in X} \|A(x_1-x_2)\|\leq 2R\|A\| $.

\end{itemize}
Therefore, we can rewrite the above inequality as follows
\begin{equation}\label{eqa}\dist^2 (x,S)+(y-y^*)^2\leq \kappa_R \left(\frac{1}{2}\|Ax-b\|^2_2+\mu y-\left(\frac{1}{2}\|Ax^*-b\|^2_2+\mu \|x^*\|_1\right) \right), \forall (x,y)\in \tilde{X},\end{equation}
where 
 \begin{equation}\label{kappa}
 \kappa_R=\nu^2 \left(2R\mu +6\left(R\|A\|+\|b\|\right)R\|A\|+2\left(R\|A\|+\|b\|\right)^2 +2\right).
 \end{equation}
 By taking $y=\|x\|_1$, \eqref{eqa} becomes
   $$\dist^2 (x,S)+(y-y^*)^2\leq \kappa_R (f(x)-f(x^*)),~\forall x\in \R^n,\: \|x\|_1\leq R.$$
We therefore obtain 
\begin{lemma}{\bf (Error bound and KL inequality for the least squares objective with $\ell^1$ regularization)}\label{sparse}
Fix $R> \frac{\|b\|^2}{2\mu}$. Then,
\begin{equation}\label{EBISTA} f(x)-f(x^*)\geq 2 \gamma_R \dist^2 (x,S) \text{ for all  $x$ in $\R^n$  such that   $\|x\|_1\leq R$, }\end{equation}
where 
 \begin{equation}\label{constant}
 \gamma_R= \frac{1}{4\nu^2 \big(1+\mu R +\left(R\|A\|+\|b\|\right)\,\left(4R\|A\|+\|b\|\right)\big)}.
 \end{equation}
 As a consequence $f$ is a KL function on the $\ell^1$ ball of radius $R$ and admits $\varphi(s)=\sqrt{2\gamma_R^{-1}\,s}$ as desingularizing function.
 \end{lemma}

\subsubsection{Distances to an intersection: convex feasibility}\label{s:feas}   
For $m\geq 2$, one considers clo\-sed con\-vex subsets $C_1, \ldots, C_m$ of $H$ whose  intersection  contains a nonempty open ball.  This proposition is a quantitative version of  \cite[Corollary 3.1]{BeckTeboulle}.
\begin{proposition}\label{deb}
Suppose that there is $\bar{x}\in H$ and $R>0$ such that 
\begin{equation}\label{ball}
B(\bar{x}, R)\subset \mathop \bigcap\limits_{i=1}^m C_i.
\end{equation}
Then,
	\begin{equation}{\label{fb1}}
		\dist(x, \mathop\cap_{i=1}^m C_i)\leq \left(1+\frac{2\|x-\bar{x}\|}{R}\right)^{m-1}\max \left\lbrace \dist(x,C_i), {i=1,\cdots,m}\right\rbrace, \:\: \forall x \in H.
	\end{equation}
		
\end{proposition}

\begin{proof}
We assume $m=2$ in a first stage. Put $C=C_1\cap C_2$,  $d=2\max \left\lbrace \dist(x,C_1),\dist(x,C_2)\right\rbrace$ and  fix $x\in H$. The function $\dist(\cdot,C_2)$ is  Lipschitz  continuous with constant 1. 
Thus, $$|\dist(P_{C_1}(x), C_2)-\dist(x,C_2)|\leq \|x-P_{C_1}(x)\|$$ and so
$$\dist(P_{C_1}(x),C_2)\leq \dist(x,C_1)+\dist(x,C_2)\leq d.$$
By taking $y=\bar{x}+\frac{R}{d} (P_{C_1}(x)-P_{C_2}P_{C_1}(x))$, we deduce that $y\in B(\bar x,R)\subset C_1\cap C_2$. Now, we construct a specific point $z\in C$ as follows
	$$z=\frac{d}{R+d}\,y+\frac{R}{R+d}P_{C_2}P_{C_1}(x).$$
Obviously $z$ is in $C_2$, and if we replace $y$ in $z$ by $\bar{x}+\frac{R}{d} \big(P_{C_1}(x)-P_{C_2}P_{C_1}(x)\big)$, we obtain
	$$z=\frac{d}{R+d}\bar{x}+\frac{R}{R+d}P_{C_1}(x) \in C_1,$$
This implies that $z\in C_1\cap C_2$. Therefore 
	$$\dist(x,C)\leq \|x-z\|\leq \|x-P_{C_1}(x)\|+\|z-P_{C_1}(x)\|,$$
and, since $\bar x \in C_1\cap C_2$, 
	$$\|z-P_{C_1}(x)\|=\frac{d}{R+d}\|\bar{x}-P_{C_1}(x)\|=\frac{d}{R+d}\|P_{C_1}(\bar{x})-P_{C_1}(x)\|\leq \frac{d}{R+d}\|\bar{x}-x\|.$$	
By combining the above results, we have $\dist(x,C)\leq \frac{d}{2}+\frac{d}{R+d}\|x-\bar{x}\|,$
which gives
	\begin{equation}{\label{fb2}}
	\dist(x,C)\leq \left(1+\frac{2\|x-\bar{x}\|}{R}\right) \max \left\lbrace \dist(x,C_1), \dist(x,C_2)\right\rbrace.
	\end{equation}	 
For arbitrary $m\ge 2$, applying \eqref{fb2} for the two sets $C_1$ and $\mathop\cap\limits_{i=2}^m C_i$, we obtain
$$\dist(x,\mathop\cap_{i=1}^m C_i)\leq \left(1+\frac{2\|x-\bar{x}\|}{R}\right)\max \left\lbrace \dist(x,C_1), \dist(x,\mathop\cap_{i=2}^m C_i)\right\rbrace.$$
Repeating the process $(m-1)$ times, we obtain \eqref{fb1}.
\end{proof}

\bigskip

\noindent{\bf A potential function for the barycentric projection method.} Let $C:=\displaystyle \cap_{i=1}^m C_i$. If $C\neq\emptyset$, finding a point in $C$ is equivalent to minimizing the following convex function over~$H$
\begin{equation}\label{f:feasible}
 f(x)=\frac{1}{2}\sum_{i=1}^m \alpha_i\dist^2(x,C_i),
 \end{equation}
where $\alpha_i>0$ for all $i=1,\ldots,m$ and $\sum_1^m \alpha_i=1$. As we shall see in the next section, the gradient method applied to $f$ yields the {\em barycentric projection method} (introduced in \cite{freddy}; see also \cite{comb,BeckTeboulle}). We now provide an error bound  for $f$ under assumption \eqref{ball}.
 
 \smallskip
It is clear that $C=\argmin f=\left\lbrace x\in H:f(x)=0\right\rbrace$. Fix any $x_0\in H$. From Proposition~\ref{deb}, we obtain that $f$ has the following local error bound:

$$\dist (x,C)\leq \left(1+\frac{2\|x_0-\bar{x}\|}{R}\right)^{m-1}\left(\frac{2}{\min\limits_{i=1,\dots,m}\alpha_i} \right)^{\frac{1}{2}} \sqrt{f(x)}, \:\: \forall x\in B(\bar{x}, \|x_0-\bar{x}\|).$$ 

Combining with Theorem \ref{Theolocal}, we deduce that $f$ satisfies the \L ojasiewicz inequality on $B(\bar{x}, \|x_0-\bar{x}\|)\cap [0<f]$ with  desingularizing function $\varphi(s)=\sqrt{\frac{2}{M}s\,}$, where


\begin{equation} \label{E:M_barycentric}
M=\frac{1}{4}\left(1+\frac{2\|x_0-\bar{x}\|}{R}\right)^{2-2m}\min\limits_{i=1,\dots,m}\alpha_i.
\end{equation}

\bigskip

 \noindent
{\bf A potential function for the alternating projection method.} Assume now that $m=2$, and set $g=i_{C_1}+\frac{1}{2}\dist(\cdot,C_2)^2$ -- a function related to the alternating projection method, as we shall see in a Section~\ref{s5}. One obviously has $g(x)\geq \frac{1}{2}\big(\dist^2(x,C_1)+\dist^2(x,C_2)\big)$  for all $x\in  H$. From the above remarks, we deduce that 
$$\dist(x,C)\leq 2\left(1+\frac{2\|x_0-\bar{x}\|}{R}\right)\sqrt{g(x)}, \forall x\in B(\bar{x}, \|x_0-\bar{x}\|).$$
Hence, $g$ satisfies the \L ojasiewicz inequality on $B(\bar{x}, \|x_0-\bar{x}\|)\cap [0<g]$ with desingularizing function given by $$\varphi(s)=\sqrt{\frac{2}{M'}\,s\,},$$
where 

\begin{equation} \label{E:M'_alternating}
M'=\frac{1}{8}\left(1+\frac{2\|x_0-\bar{x}\|}{R}\right)^{-2}.
\end{equation}



\section{Complexity for first-order methods with sufficient decrease condition}
\label{s4}

In this section, we derive complexity bounds for first-order methods with a sufficient decrease condition, under a KL inequality. In what follows, we assume, as before, that $f:H\to \Rcupinf$ is a proper lower-semicontinuous convex function such that $S=\argmin f\neq\emptyset$ and $\min f=0$.\\

\subsection{Subgradient sequences} We recall, from \cite{AttBolSva}, that $(x_k)_{k\in\N}$ in $H$ is a {\em subgradient descent sequence} for $f:H\to\Rcupinf$ if $x_0\in\dom f$ and there exist $a,b>0$ such that:

\begin{itemize}
	\item [({\bf H1})] (Sufficient decrease condition) For each $k\geq1$,
		$$f(x_k)+a\|x_k-x_{k-1}\|^2\le f(x_{k-1}).$$
	\item [({\bf H2})] (Relative error condition) For each $k\geq 1$, there is $\omega_{k}\in\partial f(x_k)$ such that
		$$\|\omega_{k}\|\le b\|x_k-x_{k-1}\|.$$
\end{itemize}

We point out that an additional continuity condition $-$ which is not necessary here because of the convexity of $f$ $-$ was required in \cite{AttBolSva}.

It seems that these conditions were first considered in the seminal and inspiring work of Luo-Tseng \cite{LuoTseng}. They were used to study convergence rates from error bounds. We adopt partly their views and we provide a double improvement: on the one hand, we show how complexity can be tackled for such dynamics, and, on the other hand, we provide a general methodology that will hopefully be used for many other methods than those considered here.

The motivation behind this definition is due to the fact that such sequences are generated by many prominent methods, such as the forward-backward method \cite{LuoTseng,AttBolSva,PierreGuiJuan} (which we describe in detail below), many trust region methods \cite{AbsMahAnd}, alternating methods \cite{AttBolSva,BST}, and, in a much more subtle manner, sequential quadratic methods and a wealth of majorization-minimization methods \cite{BP,Ed}. In Section~\ref{s5}, we essentially focus on the forward-backward method because of its simplicity and its efficiency. Clearly, many other examples could be worked out.

\begin{remark}[Explicit step for Lipschitz continuous gradient]\label{lip}{\rm 
If $f$ is smooth and its gradient is Lipschitz continuous with constant $L$, then any sequence satisfying:\\ 
 $$ \text{({\bf H2'})  $\quad$ For each } k\geq 1, \, \|\nabla f(x_{k-1})\|\le b\|x_k-x_{k-1}\|,$$
also satisfies \htwo. \\
Indeed, for every $k\ge 1$, 
$$\|\nabla f(x_{k})\| \le \|\nabla f(x_{k-1})\|+\|\nabla f(x_{k})-\nabla f(x_{k-1})\|
   \le   b\|x_k-x_{k-1}\|+L\|x_k-x_{k-1}\| 
   =  (b+L)\|x_k-x_{k-1}\|.$$
}
\end{remark}

\noindent
\begin{ejem}[The  forward-backward splitting method.] {\rm The {\em forward-backward  splitting} or {\em proximal gradient} method is an important model algorithm, although many others could be considered in the general setting we provide (see \cite{AttBolSva,BST,PierreGuiJuan}). Let $g:H\to \Rcupinf$ be a proper lower-semicontinuous convex function and let $h:H\to \R$ be a smooth convex function whose gradient is  Lipschitz  continuous with constant $L$. In order to minimize $g+h$ over $H$, the forward-backward method generates a sequence $(x_k)_{k\in\N}$ from a given starting point $x_0\in H$, and using the recursion
\begin{equation} \label{E:FB}
x_{k+1}\in\argmin\left\{g(z)+\langle \nabla h(x_k),z-x_k\rangle + \frac{1}{2\lambda_{k}}\|z-x_k\|^2:\,z \in H\right\}
\end{equation} 
for $k\geq 1$. By the strong convexity, lower-semicontinuity of the argument in the right-hand side and weak topology arguments, the set of minimizers has exactly one element. On the other hand, it is easily seen that \eqref{E:FB} is equivalent to
$$x_{k+1}\in\argmin\left\{g(z)+ \frac{1}{2\lambda_{k}}\|z-(x_k-\lambda_k\nabla h(x_k))\|^2:\,z \in H\right\}.$$
Moreover, using the proximity operator defined in Subsection \ref{SS:convex_analysis}, the latter can be rewritten as
\begin{equation}\label{FB}
x_{k+1}=\Prox_{\lambda_kg}\left(x_k-\lambda_k\nabla h(x_k)\right).
\end{equation}
When $h=0$, we obtain the {\em proximal point algorithm} for $g$. On the other hand, if $g=0$ it reduces to the classical {\em explicit gradient method} for $h$. \\

We shall see that the forward-backward method generates subgradient descent sequences if the step sizes are properly chosen.

\begin{proposition}\label{P:FB_ab}
Assume now that $0<\lambda^-\le\lambda_{k}\le\lambda^+<2/L$ for all $k\in\N$. Then
$\hone$ and $\htwo$ are satisfied for the forward-backward splitting method \eqref{FB} with 
$$a=\frac{1}{\lambda^+}-\frac{L}{2}\quad\text{ and }\quad b=\frac{1}{\lambda^-}+L.$$
\end{proposition}

\begin{proof} 
Take $k\geq 0$. For the constant $a$, we use the fundamental inequality provided in \cite[Remark 3.2(iii)]{BST}:
$$g(x_{k+1})+h(x_{k+1})\leq g(x_k)+h(x_k)-\left(\frac{1}{\lambda_{k}}-\frac{L}{2}\right)\|x_{k+1}-x_k\|^2\leq g(x_k)+h(x_k)-\left(\frac{1}{\lambda^+}-\frac{L}{2}\right)\|x_{k+1}-x_k\|^2.$$
For $b$, we proceed as in Remark~\ref{lip} above. Using  the Moreau-Rockafellar Theorem, the optimality condition for the forward-backward method is given by
$$\omega_{k+1}+\nabla h(x_k)+\frac{1}{\lambda_{k}}(x_{k+1}-x_k)=0,$$
where $\omega_{k+1}\in \partial g(x_{k+1})$. Using the Lipschitz continuity of $\nabla h$, we obtain 
$$\|\omega_{k+1} +\nabla h(x_{k+1})\|\leq \left(\frac{1}{\lambda_{k}}+L\right)\|x_{k+1}-x_k\|\leq \left(\frac{1}{\lambda^-}+L\right)\|x_{k+1}-x_k\|,$$ 
as claimed.
\end{proof}

\:
If $f=g+h$ has the KL property, Theorem \ref{T:f_to_x} below guarantees the strong convergence of every sequence generated by the forward-backward method.}
\end{ejem}

Convergence of subgradient descent sequences follows readily from \cite{AttBolSva} and \cite{BST,PierreGuiJuan}. Although this kind of result has now became standard, we provide a direct proof for estimating  thoroughly  the constants at stake.

\begin{theorem}{\bf (Convergence of subgradient descent methods in a Hilbertian convex setting)}\label{T:f_to_x}
Assume that $f:H\to\Rcupinf$ is a proper lower-semicontinuous convex function which has the KL 
property on $[0<f<\bar r]$ with desingularizing function $\varphi\in  \KK(0,\bar r)$. We consider a 
subgradient descent sequence $(x_k)_{k\in\N}$ such that $f(x_0)\leq r_0<\bar r$. Then, $x_k$ converges 
strongly to some $x^*\in\argmin f$ and
\begin{equation}\label{ineq}
\|x_k-x^*\|\leq \frac{b}{a} \varphi(f(x_k)) +\sqrt{\frac{f(x_{k-1})}{a}}, \: \forall k\geq 1.   
\end{equation}
\end{theorem}

\begin{proof}  
Using $\hone$, we deduce  that the sequence $(f(x_k))_{k\in\N}$ is nonincreasing, thus $x_k \in[0\leq f<\bar{r}]$.
Denote by $i_0$ the first index $i_0\geq1$ such that $\|x_{i_0}-x_{i_0-1}\|=0$ whenever it exists. If such an $i_0$ exists, one has $\omega_{i_0}=0$, and so, $f(x_{i_0})=0$. This implies that $f(x_{i_0+1})=0$ and thus $x_{i_0+1}=x_{i_0}$ (the sequence is then stationary.)  Hence the upper bound holds provided that it has been established for all $k\leq i_0-1$ in \eqref{ineq}. A similar reasoning applies to the case when $f(x_{i_0})=0$. 

Assume first that $f(x_k)>0$ and $\|x_k-x_{k-1}\|>0$ for all $k\geq1$. 
Combining $\hone$, $\htwo$, and using the concavity of $\varphi$ we obtain
\begin{align}
\varphi(f(x_k))-\varphi(f(x_{k+1}))&\geq \varphi'(f(x_k)) \left(f(x_k)-f(x_{k+1})\right)\notag\\
&\geq \frac{a\|x_k-x_{k+1}\|^2}{b\|x_{k-1}-x_k\|}\notag\\
& \geq\frac{a}{b}\frac{\left(2\|x_k-x_{k+1}\|\|x_k-x_{k-1}\|-\|x_{k-1}-x_k\|^2\right)}{\|x_k-x_{k-1}\|}, \forall k\geq 1.\notag\\
& \geq\frac{a}{b}(2\|x_k-x_{k+1}\|-\|x_{k-1}-x_k\|), \forall k\geq 1.\label{phong}
\end{align}

This implies
$$\frac{b}{a}\big(\varphi(f(x_1))-\varphi(f(x_{k+1}))\big) +\|x_0-x_1\|
\geq \sum_{i=1}^{k}\|x_i-x_{i+1}\|, \forall k\in \N,$$
therefore, the series $\sum_{i=1}^{\infty}\|x_i-x_{i+1}\|$ is convergent, which implies, by the Cauchy criterion ($H$ is complete), that the sequence $(x_k)_{k\in \N}$ converges to some point $x^*\in H$. From $\htwo$, there is a sequence $\omega_k\in \partial f(x_k)$ which  converges to $0$. Since $f$ is convex and lower-semicontinuous, the graph of $\partial f$ is closed in $H\times H$ for the strong-weak (and weak-strong) topology. Thus $0\in \partial f(x^*)$.\\
Coming back to \eqref{phong}, we also infer
 $$\frac{b}{a}(\varphi(f(x_k))-\varphi(f(x_{k+m}))) +\|x_{k-1}-x_k\|
 \geq \sum_{i=k}^{k+m}\|x_i-x_{i+1}\|, \forall k, m\in \N.$$
Combining the latter with $\hone$ yields
 $$\frac{b}{a}(\varphi(f(x_k))-\varphi(f(x_{k+m}))) +\sqrt{\frac{f(x_{k-1})-f(x_k)}{a}} 
  \geq \sum_{i=k}^{k+m}\|x_i-x_{i+1}\|, \forall k, m\in \N.$$
Letting $m\rightarrow \infty$, we obtain
$$\frac{b}{a} \varphi(f(x_k)) +\sqrt{\frac{f(x_{k-1})-f(x_k)}{a}} 
  \geq \|x_k-x^*\|, \forall k\in \N,$$
thus 
$$\frac{b}{a} \varphi(f(x_k)) +\sqrt{\frac{f(x_{k-1})}{a}} 
  \geq \|x_k-x^*\|, \forall k\in \N.$$
The case when $\|x_k-x_{k-1}\|$ or $f(x_k)$ vanishes for some $k$ follows easily by using the argument evoked at the beginning of the proof.\end{proof}

\begin{remark}
{\rm When $f$ is twice continuously differentiable and {\em definable} (in particular, if it is semi-algebraic) it is proved in \cite{BBJ} that $\varphi(s)\geq O(\sqrt{s})$ near the origin. This shows that, in general, the ``worst" complexity is more likely to be induced by $\varphi$ rather than the square root.
}
\end{remark}

\subsection{Complexity for subgradient descent sequences}  

This section is devoted to the study of complexity for first-order descent methods of KL convex functions in Hilbert spaces. \\

Let $0<r_0<\bar r$, we shall assume that $f$ has the KL property on $[0<f<\bar r]$ with desingularizing function $\f\in \cali K(0,\bar r)$ (recall  that $\argmin f\neq\emptyset$ and $\min f=0$.). Whence
$$
\varphi'(f(x))||\partial^0 f(x)||\geq 1
$$
for all $x\in [0 <f< \bar r]$. Set $\alpha_0=\f(r_0)$ and consider the function $\psi=(\varphi\vert_{[0,r_0]})^{-1}:[0,\alpha_0]\to [0,r_0]$, which is increasing and convex. \\

The following assumption will be useful in the sequel:

\begin{itemize}
	\item [({\bf A})] The function $\psi'$ is Lipschitz continuous (on $[0,\alpha_0]$) with constant $\ell>0$ and $\psi'(0)=0$.
\end{itemize}

Intuitively, the function  $\psi$ embodies the worst-case ``profile" of $f$. As explained below, the worst-case behavior of descent methods appears indeed to be measured through $\varphi$. The assumption~({\bf A}) is definitely weak, since for interesting cases $\psi$ is flat  near $0$, while it can be chosen affine for large values (see Proposition~\ref{convex}).

We focus on algorithms that generate subgradient descent sequences, thus complying with ({\bf H1}) and~({\bf H2}).\\

\noindent {\bf A one-dimensional worst-case proximal sequence.} Set
\begin{equation}\label{comp}
\zeta=\frac{\sqrt{1+2\ell \,a \,b^{-2}}-1}{\ell}>0,
\end{equation}
where $a>0$, $b>0$ and $\ell>0$ are given in ({\bf H1}), ({\bf H2}) and ({\bf A}), respectively. Starting from $\alpha_0$, we define the {\em one-dimensional worst-case proximal sequence} inductively by
\begin{equation}\label{worst}
\alpha_{k+1}=\argmin\left\{\psi(u)+\frac{1}{2\zeta} (u-\alpha_k)^2:u\geq 0\right\}
\end{equation}
for $k\ge 0$. Using standard arguments, one sees that $\alpha_k$ is well defined and positive for each $k\ge 0$. Moreover, the sequence can be interpreted through the recursion
\begin{equation} \label{E:alpha_prox}
\alpha_{k+1}=(I+\zeta\psi')^{-1}(\alpha_k)=\prox_{\zeta\psi}(\alpha_k),
\end{equation} 
for $k\ge 0$ and where $I$ is the identity on $\R$. 
Finally, it is easy to prove that $\alpha_k$ is decreasing and converges to zero. By continuity, $\lim\limits_{k\to\infty}\psi(\alpha_k)=0$.\\




The following is one of our main results. It asserts that $(\alpha_k)_{k\in\N}$ is a {\em majorizing sequence} ``\`a la Kantorovich":

\begin{theorem}[Complexity of descent sequences for convex KL functions]\label{t:complexity} $\:$\\
Let $f:H\to\Rcupinf$ be a proper lower-semicontinuous convex function with $\argmin f\neq\emptyset$ and $\min f=0$.  Assume further that $f$ has the KL property on $[0<f<\bar r]$.  Let $(x_k)_{k\in \N}$ be a subgradient descent sequence with $f(x_0)=r_0\in (0,\bar r)$ and suppose that  assumption {\rm({\bf A})} holds {\rm (}on the interval $[0,\alpha_0]$ with $\psi(\alpha_0)=r_0${\rm )}.

Define the one-dimensional worst-case proximal sequence $(\alpha_k)_{k\in \N}$ as above\footnote{See \eqref{comp} and \eqref{worst}.}. Then, $(x_k)_{k\in \mathbb{N}}$ converges strongly to some minimizer $x^*$ and, moreover,
\begin{align} 
& f(x_k)  \leq  \psi(\alpha_k),  \quad\forall k\geq 0,\label{E:T_worst_prox1}\\
 &\|x_k-x^*\|  \leq \frac{b}{a} \alpha_k+\sqrt{\frac{\psi(\alpha_{k-1})}{a}}, \quad\forall k\geq 1.\label{E:T_worst_prox2}
\end{align}
\end{theorem}

\begin{proof} 
For $k\ge 1$, set $r_k:=f(x_k)$. If $r_{k}=0$ the result is trivial. Assume $r_k>0$, 
then one has also $r_j>0$ for $j=1,\dots,k$. Set $\beta_k=\psi^{-1}(r_k)>0$ and $s_k=
\frac{\beta_{k-1}-\beta_{k}}{\psi'(\beta_{k})}>0$ so that $\beta_k$ satisfies
\begin{equation} \label{E:beta_prox}
\beta_{k}=(1+s_k\psi')^{-1}(\beta_{k-1}).
\end{equation}
We shall prove that $s_k\ge\zeta$.
Combining the KL inequality and $\htwo$,  we obtain that 
$$b^2\varphi'(r_{k})^2 \|x_{k}-x_{k-1}\|^2\geq \varphi'(r_{k})^2 \|\omega_{k}\|^2\geq 1,$$
where $\omega_k$ is as in $\htwo$.
Using $\hone$ and the formula for the derivative of the inverse function, this gives 
$$
\frac{a}{b^2}\leq \varphi'(r_{k})^2(r_{k-1}-r_{k})=\frac{(\psi(\beta_{k-1})-\psi(\beta_{k}))}{\psi'(\beta_{k})^2}.
$$
We now use the descent Lemma on $\psi$ (see, for instance, \cite[Lemma 1.30]{Pey_book}), to obtain
$$\frac{a}{b^2}\le \frac{(\beta_{k-1}-\beta_{k})}{\psi'(\beta_{k})} + \frac{\ell(\beta_{k-1}-\beta_{k})^2}{2\psi'(\beta_{k})^2} =  s_k+\frac{\ell}{2} s_k^2.$$
We conclude that
\begin{equation} \label{E:Lambda_sk}
s_k\geq \frac{\sqrt{1+2\ell \,a \,b^{-2}}-1}{\ell}=\zeta.
\end{equation} 
The above holds for every $k\ge 1$ such that $r_k>0$. \\

To conclude we need two simple results on the prox operator in one dimension.\\
\medskip
\noindent{\sc Claim 1.} {\em Take $\lambda^0>\lambda^1$ and $\gamma>0$. Then 
$$(I+\lambda^0\psi')^{-1}(\gamma)<(I+\lambda^1\psi')^{-1}(\gamma).$$}

\noindent {\em Proof of Claim 1}. It  is elementary, set $\delta=(I+\lambda^1\psi')^{-1}(\gamma)\in (0,\gamma)$, one indeed has 
$(I+\lambda^0\psi')(\delta)=(I+\lambda^1\psi')(\delta)+(\lambda^0-\lambda^1)\psi'(\delta)>\gamma,$
and the result follows by the monotonicity of $I+\lambda_0\psi'.$ 

\medskip

\noindent{\sc Claim 2.} {\em  Let $(\lambda_k^0)_{k\in\N},(\lambda_k^1)_{k\in \N}$ two positive sequences such that $\lambda_k^0\geq \lambda_k^1$ for all $k\geq 0$. Define the two proximal sequences
$$\beta_{k+1}^0=(I+\lambda_k^0\psi')^{-1}(\beta_k^0), \quad \beta_{k+1}^1=(I+\lambda_k^1\psi')^{-1}(\beta_k^1),$$
with $\beta^0_0=\beta_0^1\in (0,r_0]$. Then 
$\beta_{k}^0\leq \beta_k^1$ for all $k\geq 0$.}\\

\noindent{\em Proof of Claim 2.} We proceed by induction, the first step being trivial, we assume the result holds true for $k\geq 0$. We write
$$\beta^0_{k+1} = (I+\lambda_k^0\psi')^{-1}(\beta_k^0)  \leq  (I+\lambda_k^0\psi')^{-1}(\beta_k^1)  \leq  (I+\lambda_k^1\psi')^{-1}(\beta_k^1)  =  \beta_{k+1}^1,$$
where the first inequality is due to the induction assumption (and the monotonicity of $\psi'$), while the second one follows from Claim~1. \\

We now conclude by observing that $\alpha_k,\beta_k$ are proximal sequences, 
$$\alpha_{k+1}=(I+c\psi')^{-1}(\alpha_k),\quad \beta_{k+1}=(I+s_k \psi')^{-1}(\beta_k).$$
Recalling that $s_k\geq \zeta$, one can apply Claim~2 to obtain that $\alpha_k\geq \beta_k$. And thus $\psi(\alpha_k)\geq \psi(\beta_k)=r_k$. \\
The last point follows from Theorem~\ref{T:f_to_x}.
\end{proof}

\begin{remark}[Two complexity regimes]{\rm  In many cases the function $\psi$ is nonlinear near zero and is affine beyond a given threshold $t_0>0$ (see subsection \ref{SS:EB} or Proposition~\ref{convex}). This geometry  reflects  on the convergence rate of the estimators as follows:
\begin{enumerate}
\item A  fast convergence regime is observed when $\alpha_k>t_0$.  The objective is cut down by a constant value at each step. 
 \item When the sequence $\alpha_k$ enters $[0,t_0]$, a slower and restrictive complexity regime  appears. 
 \end{enumerate}
 }
\end{remark}

\begin{remark}[Complexity with a continuum of minimizers]{\rm  We draw the attention of the reader that our complexity result {\em on the sequence} (not only on the values) holds even in the case when there is a continuum of minimizers. }
\end{remark}

\smallskip

It is obvious from the proof that the following result holds.
\begin{corollary}[Stable sets and complexity]\label{c:complexity}{\rm Let $X$ be a subset of $H$. If the set $[0<f<\bar r]$ on which $f$ has the KL property  is replaced by a more general set of the form: $\bar X=X\cap [0<f<\bar r]$ with the property that 
 $x_k\in \bar X$ for all $k\geq 0$, then the same result holds.  
}
\end{corollary}

The above corollary has the advantage to relax the constraints on the desingularizing function: the smaller the set is, the lower (and thus the better) $\varphi$ can be\footnote{Desingularizing functions for a given problem (but with different domains) are generally definable in the same o-minimal structure  thus their germs are always 
comparable. This is why the expression ``the lower" is not ambiguous in our context.}. 
There are thus some possibilities to obtain  functions $\psi$ with an improved 
conditioning/geometry, which could eventually lead to tighter complexity bounds. On 
the other hand, the stability condition $x_k\in \bar X, \;\forall k\in \N$ is generally 
difficult to obtain. \\

\medskip
We conclude by providing a study of the important case $\psi(s)=\frac{\ell}{2}s^2$.
In that case assumption {\rm({\bf A})} holds, and we obtain the following particular instance of Theorem \ref{t:complexity}:

\begin{corollary}\label{T:2_complexity} 
The assumptions and the notation are those of Theorem  \ref{t:complexity}, but we assume further that $f$ has the KL property with $\psi(s)=\frac{\ell}{2}s^2$ on $[0<f<\bar r]$. We set
\begin{equation}\label{pas}
\sigma=\ell b^{-2}.
\end{equation} In that case the complexity estimates given in Theorem~\ref{t:complexity} take the form 
\begin{eqnarray}
f(x_k) & \le & \frac{f(x_0)}{(1+2 a\sigma)^k}, \quad\forall k\geq 0,\label{E:lemma2_1}\\ 
\|x_k-x^*\|  & \leq & \left[1+\frac{1}{a\sigma\sqrt{1+\frac{1}{2a\sigma}}}\right]
\frac{\sqrt{\frac{1}{a}f(x_0)}}{\:\:\left(1+2 a\sigma\right)^{\frac{k-1}{2}}}, \quad\forall k\geq 1.\label{E:lemma2_2}
\end{eqnarray}
\end{corollary}

\begin{proof}
First, recall that the one-dimensional worst-case proximal sequence $(\alpha_k)_{k\in\N}$ is given by $\alpha_0=\f(r_0)$, and
$$\alpha_{k+1}=\argmin\left\lbrace \frac{\ell}{2} s^2+\frac{1}{2\zeta}\left(s-\alpha_k \right)^2:s\geq 0\right\rbrace$$
for all $k\ge 0$, where
$$\zeta=\frac{\sqrt{1+2\ell ab^{-2}}-1}{\ell}.$$ 
Whence,
$\alpha_{k+1}=\frac{\alpha_{k}}{(1+\ell\zeta)},$
and so,
\begin{equation} \label{E:alpha_k} 
\alpha_k=\frac{\alpha_0}{(1+\ell\zeta)^k}, \;\; \forall k\geq0.
\end{equation}
 From \eqref{E:T_worst_prox1}, we immediately deduce
$$f(x_k)\le \frac{f(x_0)}{(1+\ell\zeta)^{2k}}.$$
Finally, since
$$1+\ell\zeta=\sqrt{1+2\ell ab^{-2}}=\sqrt{1+2a\sigma},$$
we obtain \eqref{E:lemma2_1}. For \eqref{E:lemma2_2}, first observe that
\begin{equation} \label{E:lemme_aux_1}
\frac{b}{a} \alpha_k=\frac{b}{a}\frac{\alpha_0}{(1+\ell\zeta)^k}=\frac{b}{a\sqrt{\ell}}\frac{\sqrt{2f(x_0)}}{(1+\ell\zeta)^k}=\frac{b}{a\sqrt{\ell}}\frac{\sqrt{2f(x_0)}}{(1+2\ell ab^{-2})^{k/2} },
\end{equation} 
while
\begin{equation} \label{E:lemme_aux_2}
\sqrt{\frac{\psi(\alpha_{k-1})}{a}}=\sqrt{\frac{\ell\alpha_{k-1}^2}{2a}}=\sqrt{\frac{\ell\alpha_0^2}{2a(1+\ell\zeta)^{2k-2}}}=\sqrt{\frac{1+2\ell ab^{-2}}{2a}}\frac{\sqrt{2f(x_0)}}{(1+2\ell ab^{-2})^{k/2}}.
\end{equation} 
In view of \eqref{E:T_worst_prox2},  by adding \eqref{E:lemme_aux_1} and \eqref{E:lemme_aux_2} we obtain:
\begin{equation}\label{op}
\|x_k-x^*\|   \leq  \left[\frac{b}{a\sqrt{\ell}}+\sqrt{\frac{1}{2a}+\frac{\ell}{b^2}}\right]\frac{\sqrt{2f(x_0)}}{(1+2\ell ab^{-2})^{k/2}}, \quad\forall k\geq 1.
\end{equation}
To conclude, observe that 
\begin{eqnarray*}
 \left[\frac{b}{a\sqrt{\ell}}+\sqrt{\frac{1}{2a}+\frac{\ell}{b^2}}\right] & = &\sqrt{\frac{1}{2a}+\frac{\ell}{b^2}}\left[1+\frac{1}{\sqrt{\frac{a\sigma}{2}+a^2\sigma^2}}\right]\\
 & = &\sqrt{\frac{1+2a\sigma}{2a}}\left[1+\frac{1}{a\sigma\sqrt{1+\frac{1}{2a\sigma}}}\right], 
  \end{eqnarray*}
and combine this last equality with \eqref{op} to obtain the result.

\end{proof}

\begin{remark}[Constants]{\rm The constant $\sigma=\ell b^{-2}$ plays the role of a step size as it can be seen in the forthcoming examples. For smooth problems and for the classical gradient method, one has for instance $\sigma=\text{constant}\,\cdot \frac{1}{L}$ (see Section~\ref{s5} below). 
}
\end{remark}

\section{Applications:  feasibility problems, uniformly convex problems and compressed sensing}\label{s5}

In this section we apply our general methodology to derive complexity results for some keynote algorithms that are used to solve problems arising in compressed sensing and convex feasibility. We shall make a constant use of Corollary~\ref{T:2_complexity}, so let us keep in mind the notation introduced in Section \ref{s4}, especially the constants $a$, $b$ and $\ell$.

\subsection{Convex feasibility problems with regular intersection}

Let $\displaystyle \big\{C_i\big\}_{i\in \{1,\ldots,m\}}$ be a family of closed convex subsets of $H$, for which there exist $R>0$ and $\bar{x}\in H$ with 
$$B(\bar{x},R)\subset C:=\mathop\bigcap\limits_{i=1}^m C_i.$$ 

\noindent
{\bf  Barycentric Projection Algorithm.} Starting from $x_0\in H$, this method generates a sequence $(x_k)_{k\in\N}$ by the following recursion 
$$x_{k+1}=\mathop\sum\limits_{i=1}^m \alpha_i P_{C_i}(x_{k}).$$
where $\alpha_i>0$ and $\sum_{i=1}^m \alpha_i=1$.\\

Using the function $f=\frac{1}{2}\sum\limits_{i=1}^{m} \alpha_i \dist^2 (\cdot,C_i)$, studied in Subsection~\ref{s:feas}, it is easy to check that  $$\nabla f(x)=\sum\limits_{i=1}^{m} \alpha_i (x-P_{C_i}x)=x- \sum\limits_{i=1}^{m} \alpha_i P_{C_i}(x)$$
for all $x$ in $H$. Thus, the sequence $(x_k)_{k\in \mathbb{N}}$ can be described by the recursion
$$ x_{k+1}=x_k-\nabla f(x_k), \:k\geq 0.$$
Moreover, $\nabla f$ is Lipschitz continuous with constant $L=1$. It follows that $(x_k)_{k\in \NN}$ satisfies the conditions \hone\ and \htwo\ with $a=\frac{1}{2}, b=2$. It is classical to see that for any $\hat x\in C$, the sequence $\|x_k-\hat x\|$ is decreasing (see, for instance, \cite{Pey_book}). This implies that $x_k\in B(\bar{x}, \|x_0-\bar{x}\|)$ for all $k\geq0$. As a consequence, $f$ has a global desingularizing function $\f$ on $B(\bar{x}, \|x_0-\bar{x}\|)$, whose inverse is given by 
$$\psi(s)=\frac{M}{2}s^2, \:s\geq0,$$ 
where $M$ is given by \eqref{E:M_barycentric}. 
Using Theorem \ref{T:2_complexity} with $a=\frac{1}{2}$, $b=2$ and $\ell=M$, we obtain:

\begin{theorem}[Complexity of the barycentric projection method for regular intersections]
The ba\-ry\-cen\-tric projection sequence $(x_k)_{k\in \N}$ converges strongly to a point $x^*\in C$ and
\begin{align*}
 & f(x_k)\:  \leq   \:\: \frac{f(x_0)}{\left(1+\frac{M}{4}\right)^k}, \quad \forall k\geq 0,\\
&  \|x_k-x^*\|  \leq \left[1+\frac{8}{M\sqrt{1+\frac{4}{M}}}\right]\, \frac{\sqrt{2f(x_0)}}{\,\left(1+\frac{M}{4}\right)^{\frac{k-1}{2}}} \:   ,\quad \forall k\geq 1,\notag
\end{align*}
where $M$ is given by \eqref{E:M_barycentric}.
\end{theorem}


\smallskip

\noindent
{\bf Alternating projection algorithm.} We consider here the feasibility problem in the case $m=2$. The von Neuman's {\em alternating projection method} is given by the following recursion 
$$x_0\in H,\quad\hbox{and}\quad x_{k+1}=P_{C_1}P_{C_2}(x_{k})\quad\forall k\ge 0.$$

\smallskip

Let $g=i_{C_1}+\frac{1}{2}\dist^2(\cdot,C_2)$ and let $M'$ be defined as in \eqref{E:M'_alternating} (Subsection~\ref{s:feas}). The function $h=\frac{1}{2} \dist^2(\cdot,C_2)$ is differentiable and $\nabla h=I-P_{C_2}$ is  Lipschitz  continuous with constant 1. We can interpret the sequence $(x_k)_{k\in\N}$ as the forward-backward splitting method\footnote{A very interesting result from Baillon-Combettes-Cominetti \cite{BCC} establishes that for more than two sets there are no potential functions corresponding to the alternating projection method.}
$$x_{k+1}=\Prox_{i_{C_1}}(x_{k}-\nabla h(x_{k}))=P_{C_1}(x_{k}-\nabla h(x_{k})),$$
and observe that the sequence satisfies the conditions $\hone$ and $\htwo$ with $a=\frac{1}{2}$ and $b=2$. As before, the fact that $x_k\in B(\bar{x}, \|x_0-\bar{x}\|)$ for all $k\geq0$, is standard (see \cite{BB}). As a consequence, the function $g$ has a global desingularizing function $\f$ on $B(\bar{x}, \|\bar{x}-x_0\|)$ whose inverse $\psi$ is $\psi(s)=\frac{M'}{2}s^2$, where $M'$ is given by \eqref{E:M'_alternating}. 
Using Corollary~\ref{T:2_complexity} with $a=\frac{1}{2}$, $b=2$ and $\ell=M'$, we obtain:

\begin{theorem}[Complexity of the alternating projection method for regular convex sets] With no loss of generality, we assume that $x_0\in C_1$.
The sequence generated by the alternating projection method converges to a point $x^*\in C$. Moreover, $x_k\in C_1$ for all $k\ge 1$, 
\begin{align*}
& \dist(x_k,C_2) \: \leq \:\: \frac{\dist(x_0,C_2)}{\left(1+\frac{M'}{4}\right)^{\frac{k}{2}}}, \quad \forall k\geq 0,\\
 & \|x_k-x^*\|  \leq \left[1+ \frac{8}{M'\sqrt{1+\frac{M'}{4}}}\right]\, \frac{\dist(x_0,C_2)}{\,\,\left(1+\frac{M'}{4}\right)^{\frac{k-1}{2}}}, \quad \forall k\geq 1,\notag
\end{align*}
where $M'$ is given by \eqref{E:M'_alternating}.
\end{theorem}

\subsection{Uniformly convex problems}
 Let $\sigma$ be a positive coefficient. The function  $f$ is called {\em $p$-uniformly convex, or simply uniformly convex,}  if there exists $p\geq 2$ such that:
$$f(y)\geq f(x)+\langle x^*,y-x\rangle +\sigma\|y-x\|^p,$$
for all $x,y\in H$, $x^*\in \partial f(x)$. It is easy to see that $f$ satisfies the 
KL inequality on $H$ with 
$\varphi(s)=p\;\sigma^{-\frac{1}{p}}\;s^{\frac{1}{p}}$ (see \cite{AttBolRedSou}).  
For such a function we have 
$$\psi(s)=\frac{\sigma}{p^p}s^p, \:s\geq 0.$$
Fix $x_0$ in $\dom f$ and set $r_0=f(x_0)$, $\alpha_0=\psi(r_0)$. The Lipschitz continuity constant of $\psi'$  is given by $\ell=\frac{(p-1)\sigma}{p^{p-1}}\,\alpha_0^{p-2}$. 
Choose a descent method satisfying ({\bf H1}), ({\bf H2}), some examples can be found in \cite{AttBolSva,PierreGuiJuan}. Set $\zeta=\frac{\sqrt{1+2\ell \,a \,b^{-2}}-1}{\ell}$. The complexity of the method is measured by the real sequence
\begin{equation*}
\alpha_{k+1}=\argmin\left\{\frac{\sigma}{p ^p}u^p+\frac{1}{2\zeta} (u-\alpha_k)^2:u\geq 0\right\}, \:k\geq 0.
\end{equation*}
The case $p=2$ can be computed in closed form (as previously), but in general only numerical estimates are available.

Proposition~\ref{EB_Poly} shows that first-order  descent sequences for piecewise polynomial convex functions have a similar complexity structure.   This shows that error bounds or KL inequalities capture more precisely the determinant geometrical factors behind complexity than mere uniform convexity.

\subsection{Compressed sensing and the $\ell^1$-regularized least squares problem }

We refer for instance to \cite{candes} for an account on compressed sensing and an insight into its vast field of applications. We consider the cost function $f:\R^n\to\R$ given by 
$$f(x)=\mu\|x\|_1+\frac{1}{2}\|Ax-d\|_2^2,$$
where $\mu>0$,  $A\in \R^{m\times n}$ and $d\in \R^m$. 

Set $g(x)=\mu\|x\|_1$ and $h(x)=\frac{1}{2}\|Ax-d\|_2^2=\frac{1}{2}\|Ax-d\|^2$, so that $g$ is proper, lower-semicontinuous and convex, whereas $h$ is convex and differentiable, and its gradient is Lipschitz continuous with constant $L=\|A^TA\|$. Starting from any $x_0\in \R^n$, the forward-backward splitting method applied to $f$ is known as the {\em iterative shrinkage thresholding algorithm} \cite{daub}\footnote{Connection between ISTA and the forward-backward splitting method is due to Combettes-Wajs \cite{Waj}}:
$$x_{k+1}=\Prox_{\lambda_k\mu\|\cdot\|_1}\left(x_k-\lambda_{k}( A^TAx_k-A^Td)\right)\quad\hbox{for}\ k\ge 0.\leqno{\rm (ISTA)}$$
Here, $\Prox_{\lambda_k\mu\|\cdot\|_1}$ is an easily computable piecewise linear object known as the {\em soft thresholding} operator (see, for instance, \cite{Waj}). This method has been applied widely in many contexts and is known to have a complexity $O\big(\frac{1}{k}\big)$. We intend to prove here that this bound can be surprisingly ``improved" by our techniques.\\

First, recall that, according to Proposition \ref{P:FB_ab}, sequences generated by this method comply with \hone\ and \htwo, provided the stepsizes satisfy $0<\lambda^-\le\lambda_k\le\lambda^+<2/L$. Recall that the constants $a$ and $b$ can be chosen as 
\begin{equation} \label{E:ab_ISTA}
a=\frac{1}{\lambda^+}-\frac{L}{2}\quad\text{ and }\quad b=\frac{1}{\lambda^-}+L,
\end{equation} 
respectively.\\

Set $\displaystyle R=\max\Big(\frac{f(x_0)}{\mu},1+\frac{\|d\|^2}{2\mu}\Big)$. We clearly have $R>\frac{\|d\|^2}{2\mu}$, and, using the fact that $(x_k)_{k\in\N}$ is a descent sequence, we can easily verify that  $\|x_k\|_1\leq R$ for all $k\in \N$. 

From  Lemma~\ref{sparse}  we know that  the function $f$ has the KL property on $[\min f<f<\min f+r_0]\cap\{x\in\R^n:\|x\|_1\leq R\}$ with\footnote{Recall that~$r_0=f(x_0)$.}  a global desingularizing function $\f$ whose inverse $\psi$ is given by
$$\psi(s)=\frac{\gamma_{R}}{2}s^2, s\geq0$$
 where $\gamma_R$ is known to exist and is bounded from above by the constant given  in  \eqref{constant}. 

\begin{remark}[Constant step size]\label{cstep}{\rm If one makes the simple choice of a {\em constant} step size all throughout the process, namely $\lambda_{k}=d/L$ with $d\in (0,2)$, one obtains $$\zeta=\frac{\sqrt{1+\frac{d(2-d)}{L(1+d)^2}\gamma_R }\,-1}{\gamma_R}\quad\mbox{ 
and }\quad\alpha_{k}=\frac{\alpha_0}{\left(1+\frac{d(2-d)}{L(1+d)^2}\,\gamma_R\right)^{k/2}}, \quad k\geq 0.$$}
\end{remark}

\medskip

Combining the above developments with Corollary~\ref{c:complexity}, we obtain the following surprising result:

\begin{theorem}[Complexity bounds for ISTA]\label{T:ISTA}
The sequence $(x_k)_{k\in \N}$ generated by {\rm ISTA} converges to a minimizer $x^*$ of $f$, and satisfies
\begin{align}\label{estis}
& f(x_k)-\min f \:\:\leq  \: \frac{f(x_0)-\min f}{q^k},\qquad \forall k\geq 0,\\ 
&  \|x_k-x^*\|  \:\: \leq \: C \frac{\sqrt{f(x_0)-\min f}}{q^{\frac{k-1}{2}}}\qquad \forall k\geq 1,
 \end{align}
where
$$q=1+\frac{2a\gamma_R}{b^2}\quad\hbox{and}\quad C=\frac{1}{\sqrt{a}}\left(1+\frac{1}{ab^{-2}\gamma_R\sqrt{1+\frac{1}{2ab^{-2}\gamma_R}}}\right).$$ 
\end{theorem}

\begin{remark}[Complexity and convergence rates for ISTA]{\rm 
(a) While it was known that ISTA has a linear asymptotic convergence rate, see \cite{jalal} in which a transparent explanation is provided, best known {\em complexity bounds} were of the type $O(\frac{1}{k})$, see \cite{BT08,drori}. Much like in the spirit of \cite{jalal}, we show here how geometry impacts complexity --through error bounds/KL inequality--  providing thus complementary results to what is usually done in this field. \\
(b) The estimate of $\gamma_R$ given in Section~\ref{s:pol} is far from being optimal and more work remains to be done to obtain acceptable/tight bounds. Observe however that the role of an optimal $\gamma_R$ is absolutely crucial when it comes to complexity (see \eqref{estis}):  a good ``conditioning" ($\gamma_R$ not too small)  provides fast convergence, while a bad one\footnote{Bad conditioning are produced by flat objective functions yielding thus   small constants $\gamma_R$.} comes with  ``bad complexity".\\
(c) Assuming that the forward-backward method is performed with a constant stepsize $d/L$ as in Remark~\ref{cstep}, the value $q$ appearing in the complexity bounds given by Theorem \ref{T:ISTA} becomes
$$q=1+\frac{d(2-d)}{(d+1)^2L}\gamma_R.$$
This quantity is maximized when $d=1/2$. In this case, one obtains the optimized estimate:
\begin{eqnarray*} 
f(x_k)-\min f & \leq & \frac{f(x_0)-\min f}{\left(1+\frac{\gamma_R}{3L}\right)^{k}},\qquad \forall k\geq 0,\\
\|x_k-x^*\| & \leq & \sqrt{\frac{2}{3L}}\left( 1+\frac{6L}{\gamma_R\sqrt{1+\frac{3L}{\gamma_R}}}\right) \frac{\sqrt{f(x_0)-\min f}}{\left(1+\frac{\gamma_R}{3L}\right)^{\frac{k-1}{2}}}
,\qquad \forall k\geq 1.
\end{eqnarray*}


}
\end{remark}

\section{Error bounds and KL inequalities for convex functions: additional properties}\label{s:theory}

In this concluding section we provide further theoretical perspectives that will help the reader to understand the possibilities and the limitations of our general methodology. We give, in particular, a counterexample to the full equivalence between the KL property and error bounds, and we provide a globalization result for desingularizing functions.

\if{
We consider the theorem following:
\begin{theorem}{\label{KL}}
Let $ x^*\in S$, $R\in]0,+\infty[$ and $\varphi\in\cali K(0,r_0)$
\begin{description}
\item[(i)] If $f$ has the KL \ property at  $x^*$, that is 
	$$\f'(f(x))\|\partial^0 f(x)\|_-\ge 1, \forall x\in  B(x^*,R)\cap[0<f<r_0],$$
then f has local error bound global on the set $B(x^*,R)\cap[0<f<r_0]$.
\item[(ii)]Conversely, if $f$ has local error bound on the set $B(x^*,R)\cap[0<f<r_0]$, in which the residual function $\psi$ satisfies
	 $$\int_{0}^{r} \frac{\psi^{-1}(s)}{s}ds<\infty,$$
 then $f(x)$ has the KL \ property at $\bar x$.
\end{description}
\end{theorem}

\begin{proof}
$(i)$ Take $x\in B(x^*,R)\cap[0<f<r_0]$, we denote $\chi_x:[0,\infty)\to \R^n$ as the solution of the differential inclusion 
	$$-\overset{.}{u}(t) \in\partial f\left(u(t)\right),$$
with an initial condition $u(0)=x$.

The function $t\mapsto\|u(t)-a\|^2$ is decreasing, $\forall a\in S$, so $u(t)\in S,\forall t>0$. By derivating the function $t\mapsto \f\left(f(\chi_x(t))\right)$, we have
	$$-\frac{d}{dt}\f\left(f(\chi_x(t))\right)=\f'\left(f(\chi_x(t))\right)\|\partial^0 f(\chi_x(t))\|\|\overset{.}\chi_x (t)\|\geq \|\overset{.}\chi_x (t)\|.$$
Integrating both terms on the interval $[0, t]$, we deduce
	$$\f\left(f(x)\right)-\f\left(f(\chi_x(t))\right)\geq \|\chi_x(t)-x\|.$$
As $t\to \infty$, and note that $\chi_x(t)\xrightarrow{t\to \infty} x^*\in S, f(\chi_x(t))\xrightarrow{t\to \infty} f(x)=0 $, we obtain 
	$$ \f(f(x))\geq \dist(x,S), \forall x\in B(x^*,R)\cap[0<f<r_0],$$
we deduce
 	$$ f(x)\geq \psi(\dist(x,S)), \forall x\in B(x^*,R) \cap[0<f<r_0],$$
 where $\psi=\f^{-1}$, which is conclusion $(ii)$.

(ii) Take $x\in B(x^*,R)\cap[0<f<r_0]$ and $p\in\partial f(x)$. By convexity, we have
 $$\langle p,x-y\rangle\geq f(x)-f(y),\forall y\in \R^n.$$
Replace $y$ by $P_Sx$, we obtain 
	 $$\|p\|\dist(x,S)\geq f(x).$$
In addition, from (ii), we deduce $\psi^{-1}(f(x))\geq \dist(x,S)$, together with the above, it follows that
	 $$\psi^{-1}(f(x))\|\partial^0 f(x)\|\geq  f(x),$$
 and so
	 $$\varphi'(f(x))\|\partial^0 f(x)\| \geq 1, \forall x\in B(x^*,R)\cap[0<f<r_0],$$
 where $\displaystyle \f(s)=\int_0^s \frac{\psi^{-1}(t)}{t}dt$, 
 as required.
 \end{proof}
}\fi

\subsection{KL inequality and length of subgradient curves}

This subsection essentially recalls a characterization result from \cite{BolDanLeyMaz} on the equivalence between the KL inequality and the existence of a uniform bound for the length of subgradient trajectories verifying a subgradient differential inclusion. Due to the contraction properties of the semi-flow,  the result is actually stronger than the nonconvex results provided in \cite{BolDanLeyMaz}. For the reader's convenience, we provide a self-contained proof.\\

Given $x\in\overline{\dom \partial f}$, we denote by $\chi_x:[0,\infty)\to H$ the unique solution of the differential inclusion
$$\dot y(t)\in-\partial f(y(t)),\mbox{ almost everywhere on }(0,+\infty),$$
with initial condition $y(0)=x$.\\

The following result provides an estimation on the length of subgradient trajectories, when $f$ satisfies the KL  inequality. Given $x\in\overline{\dom f}$, and $0\le t<s$, write
$$\length(\chi_x,t,s)=\int_{t}^{s}\|\dot{\chi}_x(\tau)\|\,d\tau.$$
Recall that $S=\argmin f$ and that $\min f=0$.

\begin{theorem}[KL and uniform bounds of subgradient curves]\label{P:1}
Let $\bar x\in S$, $\rho>0$ and $\varphi\in\cali K(0,r_0)$. The following are equivalent:
\begin{itemize}
	\item [i)] For each $y\in B(\bar x,\rho)\cap[0<f<r_0]$, we have $$\varphi'(f(y))\|\partial^0 f (y)\|\ge 1.$$
	\item [ii)] For each $x\in B(\bar x,\rho)\cap[0<f\le r_0]$ and $0\le t<s$, we have
	$$\length(\chi_x,t,s)\le \varphi\big(f\left(\chi_x\left(t\right)\right)\big)-\varphi\big(f\left(\chi_x\left(s\right)\right)\big).$$
\end{itemize}
Moreover, under these conditions, $\chi_x(t)$ converges strongly to a minimizer as $t\to\infty$.
\end{theorem}

\begin{proof}
Take $x\in B(\bar x,\rho)\cap[0<f\le r_0]$ and $0\le t<s$. First observe that 
$$\varphi(f(\chi_x(t)))-\varphi(f(\chi_x(s)))=\int_{s}^{t}\frac{d}{d\tau}\,\varphi(f(\chi_x(\tau)))\,d\tau=\int_{t}^{s} \varphi'(f(\chi_x(\tau)))\|\dot{\chi}_x(\tau)\|^2\,d\tau.$$
Since $\chi_x(\tau)\in \dom \partial f\cap B(\bar x,\rho)\cap[0<f<r_0]$ for all $\tau>0$ (see Theorem~\ref{P:gradient_curve}) and $-\dot{\chi}_x(\tau)\in\partial f(\chi_x(\tau))$ for almost every $\tau>0$, it follows that $$1\le \|\partial^0(\varphi\circ f)(\chi_x(\tau))\|\le \varphi'(f(\chi_x(\tau)))\|\dot{\chi}_x(\tau)\|$$ for all such $\tau$. Multiplying by $\|\dot{\chi}_x(\tau)\|$ and integrating from $t$ to $s$, we deduce that
$$\length(\chi_x,t,s)\le \varphi(f(\chi_x(t)))-\varphi(f(\chi_x(s))).$$
Conversely, take $y\in \dom \partial f\cap B(\bar x,\rho)\cap[0<f<r_0]$ (if $y$ is not in $\dom\partial f$ the result is obvious). For each $h>0$ we have
$$\frac{1}{h}\int_0^h\|\dot{\chi}_y(\tau)\|\,d\tau\le-\frac{\varphi(f(\chi_y(h))-\varphi(f(y))}{h}.$$
As $h\to 0$, we obtain
$$\|\dot\chi_y(0^+)\|\le\varphi'(f(y))\|\dot\chi_y(0^+)\|^2=\varphi'(f(y))\,\|\partial^0f(y)\|\,\|\dot\chi_y(0^+)\|,$$
and so
$$\|\partial^0(\varphi\circ f)(y)\|\ge 1.$$
Finally, since $\|\chi_x(t)-\chi_x(s)\|\le \length(\chi_x,t,s)$, we deduce from ii) that the function $t\mapsto \chi_x(t)$ has the Cauchy property as $t\to\infty$.
\end{proof}

\subsection{A counterexample: error bounds do not imply KL}

In \cite[Section 4.3]{BolDanLeyMaz}, the authors build a twice continuously differentiable convex function $f:\R^2\to\R$ which does not have the KL property, and such that $S=\overline{D}(0,1)$ (the closed unit disk of radius $1$). This implies that $f$ does not satisfy the KL  inequality whatever choice of desingularizing function $\varphi$ is made.

Let us show that this function has a smooth error bound. First note that, since $S$ is compact, $f$  is  coercive (see, for instance, \cite{Rockafellar}). Define $\psi:[0,\infty)\to\R_+$ by
$$\psi(s)=\min\{f(x):\|x\|\ge 1+s\}.$$
This function is increasing (recall that $f$ is convex) and it satisfies 
\begin{align}
& \psi(0)=0, \label{1}\\
 & \psi(s)>0\text{ for } s>0, \label{2}\\
& f(x)\ge \psi(\dist(x,S))\text{ for all $x\in[r<f]$} \label{3}
\end{align}
Let $\hat\psi$ be the convex envelope of $\psi$, that is the greatest convex function lying below $\psi$. One easily verifies that $\hat\psi$ enjoys the same properties \eqref{1}, \eqref{2}, \eqref{3}. The Moreau envelope of the latter:
$$\R_+\ni s \to \Psi(s):=\hat\psi_1(s)=\inf\{\hat\psi(\varsigma)+\sfrac{1}{2}(s-\varsigma)^2:\varsigma\in \R\},$$
is convex, has $0$ as a unique minimizer,  is continuously differentiable with positive derivative on $\R\setminus\{0\}$, and satisfies $\Psi\leq \psi_1$ (see \cite{BauCom}). Whence,
$$f(x)\ge \Psi(\dist(x,S)) \text{ for all $x\in[0<f<r]$}.$$
We have proved the following:

\begin{theorem}[Error bounds do not imply KL]
There exists a $C^2$ convex function $f:\R^2\to \R$, which does not satisfy the KL \ inequality, but has an error bound with a smooth convex residual function. 
\end{theorem}



\begin{remark}[H\"{o}lderian error bounds without convexity] 
{\rm H\"{o}lderian error bounds do not necessa\-ri\-ly imply \L ojasiewicz inequality $-$ not even the KL inequality $-$ for nonconvex functions. The reason is elementary and consists simply in considering a function with non isolated critical values. Given $r\geq 2$, consider the $C^{r-1}$ function 
$$f(x)=\begin{cases}
		x^{2r}  \left(2+\cos\left(\frac{1}{x}\right)\right)&\text{ if } x\ne 0,\notag\\
		0 &\text{ if } x=0.\notag
	\end{cases}$$
It satisfies  $f'(0)=0$ and $f'(x)=4rx^{2r-1}+2rx^{2r-1}\cos\left(\frac{1}{x}\right)+x^{2r-2}\sin\left(\frac{1}{x}\right)$  if $x\ne0$. 
Moreover, we have $f(x)\geq x^{2r}= \dist(x,S)^{2r}$ 
for all $x\in\R$.  On the other hand, picking  $y_k=\frac{1}{2k\pi}$ and $z_k=\frac{1}{2k\pi+3\frac{\pi}{2} }$, we see that $f'(y_k)=\frac{6r}{(2k\pi)^{2r-1}}>0$ and $f'(z_k)=\frac{1}{(2k\pi+3\frac{\pi}{2})^{2r-2}}\left(\frac{4r}{2k\pi+3\frac{\pi}{2}}-1\right)<0$
for all sufficiently large $k$. Therefore, there is a positive sequence $(x_k)_{k\in\N}$ converging to zero with $f'(x_k)=0$ for all $k$. Hence, $f$ cannot satisfy the KL inequality at $0$.}
\end{remark}




\subsection{From semi-local inequalities to global inequalities}\label{locglob}

We derive here a globalization result for KL inequalities that strongly supports the Lipschitz continuity assumption for the derivative of the inverse of a desingularizing function, an assumption that was essential to derive Theorem~\ref{t:complexity}. The ideas behind the proof are inspired by \cite{BolDanLeyMaz}.

\begin{proposition}[Globalization of KL inequality -- convex case]\label{convex} 
Let $f: H\to \Rcupinf$ be a proper lower semicontinuous convex function such that $\argmin f\neq\emptyset$ and $\min f=0$. Assume also that $f$ has the KL property on $[0<f< r_{0}]$ with desingularizing function $\varphi \in \KK(0,r_0)$. Then, given $r_1\in(0,r_0)$, the  function given by 
 $$\phi(r)= \left\{\begin{array}{ll} \varphi(r) \mbox{ when }r\leq r_1\\
  \varphi(r_1)+(r-r_1)\varphi'(r_1) \mbox{ when }r\geq r_1 
  \end{array}\right.
  $$
is desingularising for $f$ on all of $H$.
\end{proposition}

\begin{proof} Let $x$ be such that $f(x)>r_1$. We would like to establish that $\|\partial^0 f(x)\|\phi'(f(x))\geq1$, thus we may assume, with no loss of generality, that $\|\partial^0 f(x)\|$ is finite. If there is $y\in [f=r_1]$ such that $\|\partial^0 f(y)\|\leq \|\partial^0 f(x)\|$, then
\begin{equation*}
\|\partial^0 f(x)\|\phi'(f(x)) =  \|\partial^0 f(x)\|\varphi'(r_1) \geq  \|\partial^0 f(y)\|\varphi'(r_1)
 = \|\partial^0 f(y)\|\varphi'(f(y)) \geq 1. 
\end{equation*}


To show that such a $y$ exists, we use the semiflow of $\partial f$. Consider the curve $t\to\chi_x(t)$ and observe that there exists $t_1>0$ such that $f(\chi_x(t_1))=r_1$, because $f(\chi_x(0))=f(x)>r_1$, $f(\chi_x(t))\to \inf f<r_1$ and $f(\chi_x(\cdot))$ is continuous. From \cite[Theorem 3.1 (6)]{HaimBrezis}, we know also that $\|\partial f^0(\chi_x(t))\|$ is nonincreasing. As a consequence, if we set $y=\chi_x(t_1)$, we obtained the desired point and the final conclusion.
\end{proof}

\medskip

One deduces easily from the above the following result, which is close to an observation already made in \cite{BolDanLeyMaz}. For an insight into the notion of {\em definability} of functions, a prominent example being semi-algebraicity, one is referred to \cite{costedef}. Recall that coercivity of a proper lower-semicontinuous convex function defined on a finite dimensional space is equivalent to the fact that $\argmin f$ is nonempty and compact.

\begin{theorem}[Global KL inequalities for coercive definable convex functions] 
Let $f:\R^n\to\R$ be proper, lower-semicontinuous, convex, definable, and such that $\argmin f $ is nonempty and compact. Then, $f$ has the KL property on $\R^n$.
\end{theorem}

\begin{proof} Take $r_0>0$ and use  \cite{BolDanLewShi07} to obtain $\varphi\in \KK(0,r_0)$ so that $f$ is KL on $[\min f <f<\min f+r_0]$. Then use the previous proposition to extend $\varphi$ on $(0,+\infty)$.
\end{proof}

\begin{remark}{\bf (Complexity of descent methods for definable coercive convex function)} {\rm The previous result implies that {\em there always exists a global measure of complexity for first-order  descent methods $\hone,\htwo$ of definable coercive convex lower-semicontinuous functions}. This complexity bound is encoded in majorizing sequences computable from a single definable function and from the initial data. These majorizing sequences are of course defined, as in Theorem~\ref{t:complexity}, by 
\begin{equation*}
\alpha_{k+1}=\argmin\left\{\varphi^{-1}(u)+\frac{1}{2\zeta} (u-\alpha_k)^2:u\geq 0\right\}, \:\alpha_0=\varphi(r_0).
\end{equation*}
or equivalently
\begin{equation*}
\alpha_{k+1}=\prox_{\zeta\varphi^{-1}}(\alpha_k), \:\alpha_0=\varphi(r_0),
\end{equation*}
where $\zeta$ is a parameter of the chosen first-order  method. 

It is a very theoretical result  yet conceptually important since it shows that the understanding and the research of complexity is guaranteed by  the existence of a global KL inequality and our general methodology. }
\end{remark}

\section{Conclusions}
In this paper, we devised a general methodology to estimate the complexity for descent methods which are commonly used to solve convex optimization problems: error bounds can be employed to obtain desingularizing functions in the sense of \L ojasiewicz, which, in turn, provide the complexity estimates. These techniques are applied to obtain new complexity results for ISTA in compressed sensing, as well as barycentric and alternating projection method for convex feasibility.

While this work was in its final phase, we discovered the prepublication \cite{LMNP} in which  complementary ideas are used to develop error bounds for parametric polynomial systems and to analyze the convergence rate of some first order methods. Numerous interconnections and roads must be investigated at the light of these new discoveries, and we hope to do so in our future research.

\bigskip

\noindent
{\bf Acknowledgements } The authors would like to thank Amir Beck, Patrick Combettes,  \'Edouard Pauwels, Marc Teboulle and the anonymous referee for very useful comments.

\end{document}